\documentclass[11pt]{amsart}
\usepackage{amsfonts, amsbsy, amsmath, amssymb, mathtools}
\PassOptionsToPackage{hyphens}{url}\usepackage{hyperref}
\usepackage{amscd}
\usepackage{color,enumerate}
\newtheorem{thm}{Theorem}[section]
\newtheorem{lem}[thm]{Lemma}

\newtheorem{prop}[thm]{Proposition}

\newtheorem{thm-con}[thm]{Theorem-Conjecture}
\numberwithin{equation}{section}

\theoremstyle{definition}

\begin{document}
\title[Dembowski-Ostrom polynomials and reversed Dickson polynomials]{Dembowski-Ostrom polynomials and reversed Dickson polynomials}

\author[N. Fernando]{Neranga Fernando}
\address{Department of Mathematical Sciences,
	Carnegie Mellon University, Pittsburgh, PA 15213, USA}
\email{fneranga@andrew.cmu.edu}

\author[S. U. Hasan]{Sartaj Ul Hasan}
\address{Department of Mathematics, Indian Institute of Technology Jammu, Jammu 181221, India}
\email{sartaj.hasan@iitjammu.ac.in}

\author[M. Pal]{Mohit Pal}
\address{Department of Mathematics, Indian Institute of Technology Jammu, Jammu 181221, India}
\email{2018RMA0021@iitjammu.ac.in}

\begin{abstract}
We discuss the problem of classifying Dembowski-Ostrom polynomials from the composition of reversed Dickson polynomials of arbitrary kind  and monomials over finite fields of odd characteristic. Moreover, by using a variant of the Weil bound for the number of points of affine algebraic curves over finite fields, we discuss the planarity of all such Dembowski-Ostrom polynomials. Planar Dembowski-Ostrom polynomials have applications in many areas including cryptography and coding theory. 
 
\end{abstract}



\maketitle

\section{Introduction}\label{S1}
Denote, as usual, by $\mathbb{F}_q$ the finite field with $q=p^e$ elements, where $p$ is an odd prime number and $e$ is a positive integer, and by $\mathbb F_q^*$ the multiplicative group of non-zero elements of $\mathbb F_q$. For any nonnegative integer $k$, the $k$-th Dickson polynomial of the first kind $D_{k}(X,a)$ over $\mathbb F_q$ was introduced by Dickson \cite{LED} in 1897, and is defined as follows $$D_{k}(X,a) := \sum_{i=0}^{\lfloor\frac k2\rfloor}\frac{k}{k-i}\dbinom{k-i}{i}(-a)^{i}X^{k-2i},$$ where $a\in \mathbb{F}_q$ is a parameter and $D_0(X,a)=2$. More than two decades later, Schur \cite{IS} introduced a variant of Dickson polynomial of the first kind in 1923, which is now known as Dickson polynomial of the second kind. For any nonnegative integer $k$, the $k$-th Dickson polynomial of the second kind $E_{k}(X,a)$ over $\mathbb F_q$ is defined as follows $$E_{k}(X,a) := \sum_{i=0}^{\lfloor\frac k2\rfloor}\dbinom{k-i}{i}(-a)^{i}X^{k-2i},$$ where $a\in \mathbb{F}_q$ is a parameter and $E_0(X,a)=1$. A trigonometric approach for Dickson polynomials has been recently considered by Lima and Panario~\cite{LP}. Dickson polynomials have also been used to study the $c$-differential uniformity of some functions over finite fields~\cite{HPRS}. Dickson polynomials of the first and second kind over $\mathbb F_q$ were studied extensively, especially with respect to their permutation behaviour. We recall that a polynomial $f \in \mathbb F_q[X]$ is a permutation polynomial over $\mathbb F_q$ if the associated mapping $X \mapsto f(X)$ is a bijection from $\mathbb F_q$ to $\mathbb F_q$. For a non-zero element $a$ in $\mathbb{F}_q$, N\"obauer \cite{WN} proved that the Dickson polynomial of the first kind $D_{k}(X,a)$ permutes the elements of $\mathbb{F}_q$ if and only if $(k, q^2-1)=1.$ However, except for a few cases, the permutation behaviour of Dickson polynomials of the second kind $E_{k}(X,a)$ remains unresolved. One may refer to the monograph \cite{Dickson-book} for more on Dickson polynomials.

The notion of $k$-th reversed Dickson polynomial (RDP) of the first kind was introduced by Hou, Mullen, Sellers and Yucas \cite{HMSY} by simply reversing the roles of the variable $X$ and the parameter $a$ in the $k$-th Dickson polynomial of the first kind $D_{k}(X,a)$. Moreover, the authors showed that the reversed Dickson polynomials of the first kind are closely related to what are known as almost perfect nonlinear (APN) functions. The permutation behaviour of RDPs is also an interesting area of research. Hou and Ly \cite{HL} gave necessary conditions for RDPs of the first kind to be permutation polynomials over $\mathbb{F}_{q}$.     

For any nonnegative integers $k$ and $m$, the notion of $k$-th Dickson polynomial of the $(m+1)$-th kind, denoted as $D_{k,m}(X,a)$, was introduced by Wang and Yucas \cite{WY}, and is defined as follows
\begin{equation}\label{E1.1}
D_{k,m}(X,a) := \sum_{i=0}^{\lfloor\frac k2\rfloor}\frac{k-mi}{k-i}\dbinom{k-i}{i}(-a)^{i}X^{k-2i},
\end{equation}
where $0 \leq m \leq p-1$, $a \in \mathbb{F}_q$ and $D_{0,m}(X,a)=2-m$. The $k$-th RDP of the $(m+1)$-th kind is also defined in a similar way by just reversing the role of the variable $X$ and the parameter $a$ in (\ref{E1.1}). More precisely, for any nonnegative integers $k$ and $m$, the $k$-th RDP of the $(m+1)$-th kind $D_{k,m}(a,X)$ is defined as follows
\begin{equation}\label{E1.1r}
D_{k,m}(a,X) := \sum_{i=0}^{\lfloor\frac k2\rfloor}\frac{k-mi}{k-i}\dbinom{k-i}{i}(-X)^{i}a^{k-2i},
\end{equation}
where $0 \leq m \leq p-1$, $a\in \mathbb{F}_q$ and $D_{0,m} (a,X) = 2-m$. The $k$-th RDP of the $(m+1)$-th kind also satisfies the following recurrence relation
\begin{equation}\label{E1.2}
D_{k,m}(a,X)=mE_{k}(a,X)-(m-1)D_{k}(a,X).
\end{equation}
It may be noted that the permutation behaviour of RDPs of the $(m+1)$-th kind has been studied by Fernando \cite{NF}.

A function $f: \mathbb{F}_q \rightarrow \mathbb{F}_q$ is called planar if the mapping $X \mapsto f(X+ \epsilon)-f(X)-f(\epsilon)$ induces a bijection from $\mathbb{F}_q$ to $\mathbb{F}_q$ for each $\epsilon \in \mathbb{F}_{q}^*.$ Since any function from a finite field $\mathbb{F}_q$ to itself can be represented by a polynomial of degree at most $(q-1)$, one may simply consider the planarity of polynomials over finite fields. It is clear from the definition of a planar function itself that there is no planar function in the even characteristic as $X$ and $X+\epsilon$ have the same image under the mapping $X \mapsto f(X+\epsilon) -f(X) - f(\epsilon).$ A polynomial $f \in \mathbb F_q[X]$ is called exceptional planar if it is planar over $\mathbb F_{q^e}$ for infinitely many $e$. Planar functions are very important due to their wide range of applications. For example, planar functions are used to construct finite projective planes \cite{DO}, relative difference sets \cite{GS} and error-correcting codes \cite{CDY}.

A Dembowski-Ostrom (DO) polynomial over finite field $\mathbb{F}_q$ is a polynomial that admits the following shape
$$\sum_{i,j}a_{ij}X^{p^i+p^j},$$ where $a_{ij} \in \mathbb{F}_q$. DO polynomials have been used in designing a public key cryptosystem known as HFE \cite{Patarin-1996}. Note that DO polynomials provide a very rich source of planar functions. It was conjectured by R{\'o}nyai and Sz{\"o}nyi \cite{RS} (see also \cite[Conjecture 9.5.19]{HB}) that all planar functions are of ``DO type". This conjecture is still open except  in the case of characteristic $3$ for which a counter example was given by Coulter and Matthews \cite{CMLB}. In 2010, Coulter and Matthews \cite{CM} classified DO polynomials from Dickson polynomials of the first and second kind and they also discussed the planarity of such DO polynomials.

In 2016, Zhang, Wu and Liu \cite{Zhang-Wu-Liu-2016} classified DO polynomials from RDPs of the first kind in the even characteristic case and they also characterized APN functions among all such DO polynomials. In this paper, we shall extend the results of Zhang, Wu and Liu \cite{Zhang-Wu-Liu-2016} to the odd characteristic case. In fact, we give a complete classification of  DO polynomials arising from the composition of RDPs of the $(m+1)$-th kind and the monomial $X^d$, where $d$ is a positive integer, in odd characteristic, and we further characterize planar functions among these DO polynomials. The motivation behind considering this composition actually stems from the known fact that the exceptional planar polynomials $X^{10}\pm X^6-X^2$ are essentially the composition of the Dickson polynomials $D_5(X, \pm1)$ and the monomial $X^2$. DO polynomials do not have any constant term. We shall, therefore, consider the polynomials $\mathfrak{\widehat D}_{k,m} := D_{k,m}(a,X^d)-D_{k,m}(a,0)$ for the purpose of classifying DO polynomials. Notice that $\mathfrak{\widehat D}_{k,m}$ is given by  $$\mathfrak{\widehat D}_{k,m} = \sum_{i=1}^{\lfloor\frac k2\rfloor}\frac{k-mi}{k-i}\dbinom{k-i}{i}(-X^d)^{i}a^{k-2i}.$$ 
For the sake of simplicity, we shall denote $\mathfrak{\widehat D}_{k,0}$, $\mathfrak{\widehat D}_{k,1}$, $\mathfrak{\widehat D}_{k,2}$, $\mathfrak{\widehat D}_{k,3}$, and $\mathfrak{\widehat D}_{k,4}$ by $\mathfrak{\widehat D}_{k}$, $\mathfrak{\widehat E}_{k}$, $\mathfrak{\widehat F }_{k}$, $\mathfrak{\widehat G}_{k}$ and $\mathfrak{\widehat H}_{k}$, respectively.  The paper has been organized as follows. In Section \ref{S2}, we state some lemmas that will be used in the subsequent sections. In Sections \ref{S3}, \ref{S4}, \ref{S5}  and \ref{S6}, we classify DO polynomials from $\mathfrak{\widehat D}_k$, $\mathfrak{\widehat E}_k$, $\mathfrak{\widehat G}_k$ and $\mathfrak{\widehat H}_k$, respectively. The case $m\geq5$ has been considered in Section \ref{S7}. In Section \ref{S8}, we consider the planarity of DO polynomials obtained in the previous sections. The complete list of DO polynomials derived from reversed Dickson polynomials is given in Appendix A.    

Throughout the paper, we always assume that $p$ is an odd prime, $d$ is a positive integer, and $i, j, k, \ell, m, n,s, t, \alpha, \beta, \gamma, \delta$ are nonnegative integers unless specified otherwise. The greatest common divisor of positive integers $a$ and $b$ is denoted by $(a,b)$. 

\section{Some useful lemmas}\label{S2}
As alluded earlier, we shall first classify DO polynomials derived from the composition of RDPs of the $(m+1)$-th kind and the monomial $X^d$, where $d$ is a positive integer. Since DO polynomials are closed under the composition with the monomial $X^{p}$, it would be sufficient to consider the cases when $(d,p)=1$. One may also note that the monomial $X^{rd}$, where $r$ is positive integer, is a DO monomial if and only if $rd=p^{\beta}(p^{\alpha}+1)$ for some nonnegative integers $\alpha$ and $\beta$. Here, $\beta$ is the highest exponent of $p$ such that $p^{\beta} \mid r$. It is obvious that whenever $(r,p)=1$, we must have $\beta =0$. In what follows, we shall invoke these assumptions and conventions as and when required.

We now present some lemmas which will be useful in the sequel.
\begin{lem}\label{L2.1M4}
Let $d$ be a positive integer and $p>3$ be a prime such that $(d,p)=1$. Assume that the coefficients of $X^d$ and $X^{2d}$ in the polynomial $\mathfrak{\widehat D}_{k,m}$ are non-zero. Then the polynomial $\mathfrak{\widehat D}_{k,m}$ is not a DO polynomial. 
\end{lem}
\begin{proof}
Assume that $p>3$ and the coefficients of $X^d$ and $X^{2d}$ in the polynomial $\mathfrak{\widehat D}_{k,m}$ are non-zero. Therefore, if $\mathfrak{\widehat D}_{k,m}$ is a DO polynomial then  $d = p^{\alpha}+1$ and $2d = p^{\beta}+1$. Thus, we have $2p^{\alpha}+1 = p^{\beta}$, which is true if and only if $\alpha = 0$, $\beta =1$ and $p=3$. This is a contradiction to our assumption that $p>3$, hence $\mathfrak{\widehat D}_{k,m}$ is not a DO polynomial.
\end{proof} 
\begin{lem}\label{L2.2M4}
Let $d$ be a positive integer and $p>5$ be a prime such that $(d,p)=1$. Assume that the coefficients of $X^d$ and $X^{3d}$ in the polynomial $\mathfrak{\widehat D}_{k,m}$ are non-zero. Then the polynomial $\mathfrak{\widehat D}_{k,m}$ is not a DO polynomial. 
\end{lem}
\begin{proof}
The proof follows using a similar reasoning as in the proof of Lemma~\ref{L2.1M4}.
\end{proof} 
\begin{lem}\label{L2.3M4}
Let $d$ be a positive integer and $p>3$ be an odd prime such that $(d,p)=1$. Assume that the coefficients of $X^{3d}$ and $X^{4d}$ in the polynomial $\mathfrak{\widehat D}_{k,m}$ are non-zero. Then the polynomial $\mathfrak{\widehat D}_{k,m}$ is not a DO polynomial. 
\end{lem}
\begin{proof}
The proof is along the similar line as in the proof of Lemma~\ref{L2.1M4}.
\end{proof}
\begin{lem}\label{L2.33M4}
Let $p=3$ and $d$ be a positive integer such that $(d,3)=1$. Assume that the coefficients of $X^{d}$ and $X^{4d}$ in the polynomial $\mathfrak{\widehat D}_{k,m}$ are non-zero. Then the polynomial $\mathfrak{\widehat D}_{k,m}$ is not a DO polynomial. 
\end{lem}
\begin{proof}
The proof follows by using a similar argument as in Lemma~\ref{L2.1M4}.
\end{proof}
Now we recall the following lemma from \cite[Proposition 6.39]{BOOK}, which will be used later.
\begin{lem}(Legendre's formula)\label{L2.4M4}
For any nonnegative integer $\omega$ and any prime $p$, $E_p(\omega!)$ the largest exponent of $p$ that divides $\omega!$ is given by  \[E_p(\omega!)=\sum_{i=1}^{\infty} \left \lfloor \frac{\omega}{p^i}\right \rfloor = \frac{\omega-\omega_s}{p-1}, \] where $\omega_s$ is the sum of the digits in the representation of $\omega$ to the base $p$.
\end{lem}
\section{DO polynomials from RDPs of the first kind}\label{S3}

Before we begin the classification of DO polynomials from RDPs of the first kind, we shall slightly deviate and prove the following proposition that readily gives DO polynomials arising from RDPs of the $(m+1)$-th kind when the parameter $a$ is zero. 
\begin{prop}
\label{prop1}
The polynomial $D_{k,m}(0,X^d)$ is DO if and only if $k$ is even, $m\not\equiv 2 \pmod p$ and $kd$ is of the form  $2p^j(p^i+1)$, where $i,j \geq 0$.
\end{prop}
\begin{proof}
 We know that $$\displaystyle D_{k,m}(0,X^d) = \begin{cases} 0 &\mbox{if } k ~\mbox{is odd}; \\
(2-m)(-X^d)^{\frac{k}{2}} & \mbox{if } k ~\mbox{is even}.\end{cases}$$
Clearly, $D_{k,m}(0,X^d)$ is a DO polynomial if and only if $k$ is even, $m\not\equiv 2 \pmod p$ and $kd=2p^j(p^i+1).$
\end{proof}
In view of Proposition~\ref{prop1}, we shall assume that $a$ is non-zero for the rest of the paper. 

We now consider RDPs of the first kind. 

For $a\neq 0$, we write $X=Y(a-Y)$ with an indeterminate $Y\in \mathbb{F}_{q^2}$. Then 
$$D_{k,0}(a,X)=Y^k+(a-Y)^k;$$
see \cite[Section 2]{NF}. Also, we have $D_{k,0}(a,0)=a^k$. Since $D_{kp,0}(a,X)=(D_{k,0}(a,X))^p$ and $D_{kp,0}(a,0)=(D_{k,0}(a,0))^p$, we have $\mathfrak {\widehat D}_{kp}=\mathfrak {\widehat D}_k^p$, where

\begin{equation*}
\mathfrak{\widehat D}_k=D_{k,0}(a,X^d)- D_{k,0}(a,0)= \sum_{i=1}^{\lfloor\frac k2\rfloor}\frac{k}{k-i}\dbinom{k-i}{i}(-X^d)^{i}a^{k-2i}.
\end{equation*}
Since $\mathfrak {\widehat D}_{kp}=\mathfrak {\widehat D}_k^p$, it would be sufficient to consider the cases when $(k,p)=1$. The following theorems give a complete classification of DO polynomials from polynomial $\mathfrak{\widehat D}_k$ for $k$ odd and $k$ even, respectively.
\begin{thm}
Let $q$ be a power of an odd prime $p$, $a \in \mathbb{F}_{q}^*$ and $k$ odd. The polynomial $\mathfrak {\widehat D}_k$ is a DO polynomial over $\mathbb F_q$ if and only if one of the following holds.
\begin{enumerate}[{\normalfont (i)}]
\item $p=3$, $d=2p^n$, $k = 5p^{\ell},7p^{\ell}$, where $\ell, n\geq 0$.
\item $p>3$, $d=p^n(p^{\alpha}+1)$, $k=3p^{\ell}$, where $\ell, n, \alpha\geq 0$. 
\end{enumerate}
\end{thm}
\begin{proof}
The sufficiency of the theorem is straightforward. It only remains to show the necessity. Notice that when $k$ is odd, then
\begin{equation}\label{E3.1}
\begin{split}
\mathfrak{\widehat D}_{k}=-kX^{d}a^{k-2} +\frac{(k-3)k}{2!}X^{2d}a^{k-4} 
  - \frac{(k-4)(k-5)k}{3!}X^{3d}a^{k-6}+\\ \cdots +(-1)^{\frac{k-3}{2}}\,\frac{(k-1)k(k+1)}{24}\,X^{\frac{d(k-3)}{2}}a^3
+(-1)^{\frac{k-1}{2}}\,k\,X^{\frac{d(k-1)}{2}}a.
\end{split}
\end{equation}
Since $(k,p)=1$, the first term $-ka^{k-2}X^d$ in $\mathfrak{\widehat D}_k$ will always exist. Thus, if $\mathfrak{\widehat D}_k$ is a DO polynomial then $d =p^j(p^i+1)$. Since $(d,p)=1$, we have $j=0$. Therefore, we shall always take $d=p^i+1$. Now we consider two cases, $k\not\equiv 3 \pmod p$ and $k\equiv 3 \pmod p$.

\textbf{Case 1.} Let $k\not\equiv 3 \pmod p$. In this case, the coefficient of the second term in (\ref{E3.1}) is non-zero. Therefore, if $\mathfrak {\widehat D}_k$ is a DO polynomial, then $2d= p^\beta(p^\alpha +1)$ and $d=p^i+1$. Since $p$ is odd and $(d,p)=1$, $\beta=0$. Hence, the first equation reduces to $2d=p^{\alpha}+1$. Combining these two equations, we obtain $2p^i+1=p^\alpha$, which is true if and only if $p=3$, $\alpha =1$, $i=0$ and $d=2$. Therefore, in this case, we shall always assume that $p=3$ and $d=2$. For $k=5$ and $k=7$, the polynomials $\mathfrak {\widehat D}_5=a^3X^2+2aX^4$ and $\mathfrak {\widehat D}_7=2a^5X^2+2a^3X^4+2aX^6,$ are clearly DO polynomials. Now we claim that when $p=3$ and $k>7$ is odd, $\mathfrak {\widehat D}_k$ is never a DO polynomial. Since $(k,3)=1$, we have only two cases to consider, namely, $k\equiv 2\pmod{3}$ and $k\equiv 1\pmod{3}$. 

In the case $k\equiv 2\pmod{3}$, consider the second last term in (\ref{E3.1}) which is given by  $(-1)^{\frac{k-3}{2}}\,\frac{(k-1)k(k+1)}{24}\,a^3\,X^{k-3}.$ If the coefficient of the second last term in (\ref{E3.1}) is non-zero, then we claim that $(k-3)$ cannot be written as $3^i+3^j$ for some nonnegative integers $i$ and $j$. On the contrary assume that $k-3=3^i+3^j$, which implies $k-2=3^i+3^j+1$. Since $k\equiv 2\pmod{3}$, $k-2=3^i+3^j+1$ if and only if $i=j=0$. But $i=j=0$ implies $k=5$, which is a contradiction to our assumption that $k>7$. Therefore $\mathfrak{\widehat D}_k$ is not a DO polynomial in this case.  

Now assume that the coefficient of the second last term in (\ref{E3.1}) is zero. In this case, we shall show that the fourth term always exists. Note that the fourth term contains the monomial $X^8$ whose exponent cannot be written as $3^i+3^j$ for some nonnegative integers $i$ and $j$. The coefficient of the fourth term is given by  
\begin{equation}\label{4thterm}
\frac{k(k-5)(k-6)(k-7)}{24}a^{k-8}.
\end{equation}

Since $(k,3)=1$, $3\nmid k$ and $3\nmid (k-6)$. Since $k\equiv 2\pmod{3}$, where $k$ is odd and greater than $7$, $(k-5)(k-7)$ is a multiple of 24, i.e. $(k-5)(k-7)=24b$, where $b$ is an integer. Then the coefficient of the fourth term in (\ref{4thterm}) becomes $k(k-6)b$. 

Now we show that $3\nmid\, b$. On the contrary, assume that $3\mid b$. Then we have, $(k-5)(k-7)=72e$ for some integer $e$. Since $k\equiv 2\pmod{3}$, write $k=3e_1-1$ for some integer $e_1$. Recall that the second last term in \eqref{E3.1} vanishes. By substituting $3e_1-1$ for $k$ in the coefficient of the second last term, we obtain $e_1\equiv 0\pmod{3}$. Let $e_1=3n_1$ for some integer $n_1$. Then $k=3e_1-1=9n_1-1\equiv -1\pmod{9}$. From $(k-5)(k-7)=72e$ and $k\equiv -1\pmod{9}$, we have $3\equiv 0\pmod{9}$, which is a contradiction. Therefore, our assumption that $3 \mid b$ is wrong, and hence the coefficient of $X^{8}$ is non-zero. Therefore, when the second last term in \eqref{E3.1} vanishes, $\mathfrak{\widehat D}_k$ is not a DO polynomial. 

In the case $k\equiv 1\pmod{3}$, we first look at the fourth term. Recall that the fourth term contains the monomial $X^8$ whose exponent cannot be written as $3^i+3^j$ for some nonnegative integers $i$ and $j$. If the coefficient of the fourth term given in (\ref{4thterm}) is non-zero, then clearly $\mathfrak {\widehat D}_k$ is not a DO polynomial. In the case of the coefficient of the fourth term is zero, we claim that the coefficient of the 7th term, which contains the monomial $X^{14}$, is non-zero. The coefficient of the 7th term is given by $\frac{k(k-8)(k-9)(k-10)(k-11)(k-12)(k-13)}{7!}a^{k-14}.$

It is clear that the exponent of the monomial $X^{14}$ cannot be written as $3^i+3^j$ for some nonnegative integers $i$ and $j$. Since $k$ is odd and $k\equiv 1\pmod{3}$, it is clear that $3\mid (k-10)$, $6\mid (k-13)$, $9\nmid (k-10)$ and $12\nmid (k-13)$. Also, $3\nmid\,k$, $3\nmid\,(k-9)$, $3\nmid\,(k-12)$, $3\nmid(k-8)$ and $3\nmid(k-11)$. Therefore, the coefficient of the 7th term is non-zero. Hence, in the case of the coefficient of the fourth term is zero, $\mathfrak{\widehat D}_k$ is not a DO polynomial.

\textbf{Case 2.} Let $k\equiv 3 \pmod p$. In this case, notice that if $p=3$, then $k\equiv 0 \pmod 3$, which is a contradiction as $(k,p)=1$. Therefore we shall assume that $p>3$. For $k=3$, the polynomial $\mathfrak {\widehat D}_3=-3aX^d$ is a DO polynomial if and only if $d=p^i+1$. For $k>3$, consider the third term in (\ref{E3.1}), which contains the monomial $X^{3d}$. Since $k \equiv 3 \pmod p$, $k\not\equiv 4,5 \pmod p$. Hence the coefficient of the third term is non-zero. Thus, if $\mathfrak{\widehat D}_{k}$ is a DO polynomial, then $d = p^i+1$ and $3d=p^j+1$. Combining these two equations, we have $3p^i+2 = p^j$, which is true if and only if $i=0$, $j=1$, $p=5$ and $d=2$. Notice that the coefficient of last term in (\ref{E3.1}), which contains the monomial $X^{k-1}$, is non-zero. Thus, if $\mathfrak{\widehat D}_k$ is a DO polynomial then $k-1 =5^j(5^i+1)$. Since $k\equiv 3 \pmod 5$, $k \not\equiv 1 \pmod 5$, and hence $j=0$. Also, notice that if $i=0$ then $k=3$, which is a contradiction as $k>3$. Therefore $k-1=5^i+1$ which implies that $k\equiv 2 \pmod 5$, a contradiction as $k\equiv 3 \pmod 5$. Therefore for $k>3$, $\mathfrak{\widehat D}_k$ is never a DO polynomial. This completes the proof.
\end{proof}

\begin{thm}
Let $q$ be a power of an odd prime $p$, $a \in \mathbb{F}_{q}^*$ and $k$ even. The polynomial $\mathfrak {\widehat D}_k$ is a DO polynomial over $\mathbb F_q$ if and only if one of the following holds.
\begin{enumerate}[{\normalfont (i)}]     
\item $d=p^n(p^{\alpha}+1)$, $k=2p^{\ell}$, where $\ell, n, \alpha\geq 0$.
\item $p=3$, $d=2p^n$, $k = 4p^{\ell}$, where $\ell, n\geq 0$.
\end{enumerate}
\end{thm}
\begin{proof}
The sufficiency of the theorem is straightforward. It only remains to show the necessity. Notice that when $k$ is even, then
\begin{equation}
\begin{split}
\label{evencase}
\mathfrak {\widehat D}_k= -ka^{k-2}X^d+\frac{k(k-3)}{2}a^{k-4}X^{2d}-\frac{k(k-4)(k-5)}{6}a^{k-6}X^{3d}\\+\cdots 
+(-1)^{\frac{k}{2}-1}\,\frac{k^2}{4}a^2X^{d(\frac{k}{2}-1)}+(-1)^{\frac{k}{2}}\cdot 2\cdot X^{\frac{dk}{2}}.
\end{split}
\end{equation}
Since $(k,p)=1$, the first term $-ka^{k-2}X^d$ in $\mathfrak{\widehat D}_k$ will always exist. Thus, if $\mathfrak{\widehat D}_k$ is a DO polynomial, then $d =p^j(p^i+1)$. Since $(d,p)=1$, we have $j=0$. Therefore, we shall always take $d=p^i+1$. When $k=2$, the polynomials $\mathfrak {\widehat D}_2=-2X^{p^\alpha+1}$ is clearly a DO polynomial. For $k\geq 4$, we consider two cases, $k\not\equiv 3 \pmod p$ and $k\equiv 3 \pmod p$.

\textbf{Case 1.} Let $k\not\equiv 3\pmod{p}$. In this case, the coefficient of the second term in (\ref{evencase}), which contains the monomial $X^{2d}$, is non-zero. Thus, if $\mathfrak{\widehat D}_k$ is a DO polynomial, then $2d=p^{\beta}(p^{\alpha}+1)$ and $d=p^i+1$. Since $p$ is odd and $(d,p)=1$, we have $\beta = 0$. Combining these two equations, we obtain $p=3, i=0, \alpha =1$ and $d=2$. Therefore in what follows, we shall take $p=3$ and $d=2$. In the case $k=4$, the polynomial $\mathfrak {\widehat D}_4=2a^2X^2+2X^4$ is clearly DO polynomial. 

Now for $k>4$, even and $k\not\equiv 3 \pmod 3$, we claim that $\mathfrak {\widehat D}_k$ is not a DO polynomial. Consider the fourth term, which contains the monomial $X^8$. It is clear that $8$ cannot be written as $3^i+3^j$ for some nonnegative integers $i$ and $j$. If the coefficient of the fourth term in (\ref{4thterm}) is non-zero, then $\mathfrak {\widehat D}_k$ is not a DO polynomial. Now consider the case where the coefficient of the fourth term is zero. Note that the coefficient of the last term in (\ref{evencase}), which contains the monomial $X^{k}$, is always non-zero. Thus, if $\mathfrak{\widehat D}_k$ is a DO polynomial, then $k=3^{i}+1.$ Clearly, $i\neq 0$, otherwise $k=2$, a contradiction. If $i>0$, then $k\equiv 1\pmod{3}$. Now consider the second last term in (\ref{evencase}), which contains the monomial $X^{k-2}$. Clearly, the coefficient is non-zero as $(k,3)=1$. If $\mathfrak {\widehat D}_k$ is a DO polynomial and $k\equiv 1 \pmod 3$, then  $k-2=3^{j}+1$. If $i=0$ then $k=4$, which is a contradiction since $k>4$. If $i>0$, then $k=3^i+3$. This contradicts the assumption that $(k,3)=1$. Thus $\mathfrak{\widehat D}_k$ is not a DO polynomial in this case. 

\textbf{Case 2.} Let $k\equiv 3 \pmod p$. In this case, if $p=3$, then $k\equiv 0 \pmod 3$, which is a contradiction as $(k,p)=1$. Therefore, we shall always consider $p>3$. Notice that the coefficient of the last term in (\ref{evencase}), which contains the monomial $ X^{\frac{kd}{2}}$, is non-zero. Thus, if $\mathfrak{\widehat D}_k$ is a DO polynomial, then $ \frac{kd}{2}=p^{\beta}(p^{\alpha}+1)$ and $d=p^i+1$. Since $(k,p)=1$ and $(d,p)=1$, $\beta = 0$. Hence, the first equation reduces to $kd =2p^{\alpha}+2$. Combining these two equations, we get $kp^i+k =2p^{\alpha}+2$. If $i=0$, then $k=p^{\alpha}+1$, which implies $k\equiv 2\pmod{p}$ or $k\equiv 1\pmod{p}$ depending on whether $\alpha=0$ or $\alpha>0$, respectively, a contradiction. If $i>0$, then $\alpha >0$, otherwise $k(p^i+1)=4$, which is a contradiction as $p>3$. Therefore $k\equiv 2\pmod{p}$, a contradiction. This completes the proof. 
\end{proof}

\section{DO polynomials from RDPs of the second kind} \label{S4}
Recall that $D_{k,1}(a,X^d)- D_{k,1}(a,0)$ is denoted by $\mathfrak{\widehat E}_k$, where
\begin{equation} \label{E4.1}
\begin{split}
\mathfrak{\widehat E}_k=(1-k)a^{k-2}X^d+\frac{(k-2)(k-3)}{2!}a^{k-4}X^{2d}-\frac{(k-3)(k-4)(k-5)}{3!}a^{k-6}X^{3d}+ \cdots. 
\end{split}
\end{equation}
The following theorems give necessary and sufficient conditions for RDPs of the second kind to be DO polynomials for $p=3$ and $p\geq 5$, respectively.
\begin{thm} Let $q$ be a power of the odd prime $p=3$ and $a \in \mathbb{F}_{q}^*$. The polynomial $\mathfrak{\widehat E}_k$ is a DO polynomial over $\mathbb{F}_q$ if and only if one of the following holds.
\begin{enumerate}[{\normalfont (i)}]
\item $k=2,3,5,6$ and $d= p^n(p^{\alpha}+1)$, where $\alpha, ~n \geq 0.$
\item $k=4$ and $d= p^n(p^{\alpha}+1)/2$, where $\alpha, ~n \geq 0$.
\item $k = 7,10,13,19$ and $d= 2p^n$, where $n \geq 0$.
\item $k = 15$ and $d= 4p^n$, where $n \geq 0$.
\end{enumerate}
\end{thm}

\begin{proof} The sufficient part of the theorem is straightforward, therefore, we only prove the necessary part. If the polynomials $\mathfrak{\widehat E}_2= -X^d, \mathfrak{\widehat E}_3= -2aX^d$ and $\mathfrak{\widehat E}_5 =2a^3X^{d}$ are DO polynomial, then $d$ is of the form $p^{\alpha}+1$. Similarly, the polynomial $\mathfrak{\widehat E}_4=X^{2d}$ is a DO polynomial only if $d$ is of the form $(p^{\alpha}+1)/2$. If the polynomial $\mathfrak{\widehat E}_6 = a^4X^{d}+2X^{3d}$ is a DO polynomial, then $d = 3^{\alpha}+1$ and $3d = 3^t(3^{\beta}+1)$. Since $3^t \mid 3$, $t = 1$. Therefore, $\mathfrak{\widehat E}_6$ is a DO polynomial only if $d$ is of the form $p^{\alpha}+1$. The polynomial $\mathfrak{\widehat E}_7 = a^3X^{2d}+2aX^{3d}$ is a DO polynomial only if $2d = 3^{\alpha}+1$ and $3d = 3^t(3^{\beta}+1)$. Since $3^t \mid 3$, $t = 1$. Combining these two equations, we obtain $\beta = 0$, $\alpha =1$ and $d=2$. For $k\geq8$, we shall treat all possible cases depending on the value of $k$ modulo $9$.

\textbf{Case 1.} Let $k \equiv 2,8\pmod{9}$. In this case, $k\equiv2 \pmod 3$, therefore, $k \not\equiv 1\pmod{3}$ and hence the coefficient of $X^{d}$ in (\ref{E4.1}), is non-zero. Now, consider the fourth term 
\begin{equation}\label{E4th}
 \frac{(k-4)(k-5)(k-6)(k-7)}{4!}a^{k-8}X^{4d}.
\end{equation}
It is clear that $3 \nmid (k-4)$, $3 \nmid(k-6)$ and $3\nmid(k-7)$. Also, since $k\equiv 2,8 \pmod 9$, $k\not\equiv 5 \pmod 9$, and hence the highest exponent of $3$ which divides the numerator of the coefficient of fourth term is $1$. By Lemma~\ref{L2.4M4}, the highest exponent of $3$ which divides $4!$ is $1$. Therefore, coefficient of the fourth term is non-zero and $\mathfrak{\widehat E}_k$ is not a DO polynomial by  Lemma~\ref{L2.33M4}.

\textbf{Case 2.} Let $k\equiv 0,3 \pmod 9$. In this case, $k \equiv 0 \pmod 3$ and hence, the coefficient of $X^d$ in (\ref{E4.1}) is non-zero. Now consider the fourth term as given in (\ref{E4th}) again. Following similar arguments as in the Case~1 above, it is easy to see that the coefficient of the fourth term is nonzero and hence $\mathfrak{\widehat E}_k$ is not a DO polynomial by  Lemma~\ref{L2.33M4}.

\textbf{Case 3.} Let $k \equiv 1\pmod{9}$. In this case, if the polynomial $\mathfrak{\widehat E}_{10}= a^6X^{2d}+a^4X^{3d}+2X^{5d}$ is a DO polynomial, then $2d = 3^{\alpha}+1$, $3d = 3^t(3^{\beta}+1)$ and $5d = 3^{\gamma}+1$. Since $3^t \mid 3$, $t=1$. Combining the first two equations, we obtain $\beta =0$, $\alpha = 1$ and $d =2$. Now, putting these values in third equation, we have $3^{\gamma} = 9$ and $\gamma = 2$. Similarly, if the polynomial $\mathfrak{\widehat E}_{19} =a^{15}X^{2d}+a^{13}X^{3d}+2a^{9}X^{5d}+2aX^{9d}$ is a DO polynomial, then $2d = 3^{\alpha}+1$, $3d = 3^t(3^{\beta}+1)$, $5d = 3^{\gamma}+1$ and $9d = 3^s(3^{\delta}+1)$. Since $3^t \mid 3$ and $3^s \mid 9$, we have $t=1$ and $s=2$. Combining first, second and fourth equation, we obtain $\beta =0$, $\alpha = 1$ and $d =2$. Now, putting these values in third equation, we have $3^{\gamma} = 9$ and $\gamma = 2$. For $k\geq28$, since $k \equiv 1\pmod{3}$, $k \not\equiv 0,2\pmod{3}$, and hence the coefficient of $X^{2d}$ is non-zero. Now, consider the $11$th term 
\begin{equation}\label{11th}
\frac{(k-11)(k-12)(k-13)\cdots(k-19)(k-20)(k-21)}{11!}a^{k-22}(-X^d)^{11}.
\end{equation}
By Lemma~\ref{L2.4M4}, the highest exponent of $3$ that divides $11!$ is $4$. In the numerator of the coefficient of $11$th term, $(k-13), (k-16), (k-19) \equiv 0\pmod{3}$ and $(k-13), (k-16) \not\equiv 0\pmod{9}$. Now, if $k \not\equiv 19\pmod{27}$, then the highest exponent of $3$ which divides the numerator is $4$. Hence the coefficient of $X^{11d}$ is non-zero. Thus, if $\mathfrak{\widehat E}_k$ is a DO polynomial then $2d = 3^{\alpha}+1$ and $11d = 3^{\beta}+1$. Combining these equations, we have $11\cdot 3^{\alpha}+9 = 2\cdot 3^{\beta}$, which forces $\alpha = 2$ and $3^\beta = 54$, a contradiction. Therefore, $\mathfrak{\widehat E}_k$ is not a DO polynomial in this case. In the case $k\equiv 19 \pmod {27}$, we have $k\geq 46$. In this case, consider the $20$th term 
\begin{equation}\label{20th}
 \frac{(k-20)(k-21)(k-22)\cdots(k-37)(k-38)(k-39)}{20!}a^{k-40}X^{20d}.
 \end{equation}
The arguments of Case 1 can be invoked here to shows that the coefficient of $X^{20d}$ is non-zero. Thus, if $\mathfrak{\widehat E}_k$ is a DO polynomial, then $2d = 3^{\alpha}+1$ and $20d = 3^{\beta}+1$. Combining these equations, we have $10\cdot 3^{\alpha}+9 = 3^{\beta}$, which forces $\alpha = 2$ and $3^\beta = 99$, a contradiction. Therefore, $\mathfrak{\widehat E}_k$ is not a DO polynomial in this case.

\textbf{Case 4.} Let $k \equiv 4\pmod{9}$. In this case, if the polynomial $\mathfrak{\widehat E}_{13} = a^9X^{2d}+a^3X^{5d}+aX^{6d}$ is a DO polynomial, then $2d = 3^{\alpha}+1$, $5d = 3^{\beta}+1$ and $6d = 3^t(3^{\gamma}+1)$. Since $3^t \mid 6$, $t=1$. Combining these equations, we obtain $\alpha =1$, $\beta = 2$ and $d =2$. Now, for $k\geq 22$, since $k \equiv 1\pmod{3}$, we have $k \not\equiv 0,2 \pmod 3$, and hence the coefficient of $X^{2d}$ in (\ref{E4.1}) is non-zero. Now, consider the $11$th term as given in (\ref{11th}). By Lemma~\ref{L2.4M4}, the highest exponent of $3$ that divides $11!$ is $4$. In the numerator of the coefficient of $11$th term, $(k-13), (k-16), (k-19) \equiv 0\pmod{3}$ and $(k-16), (k-19) \not\equiv 0\pmod{9}$. Now if $k \not\equiv 13\pmod{27}$, then the highest exponent of $3$ which divides the numerator is $4$. Hence the coefficient of $X^{11d}$ is non-zero. Thus if $\mathfrak{\widehat E}_k$ is a DO polynomial, then $2d = 3^{\alpha}+1$ and $11d = 3^{\beta}+1$. Combining these two equations, we have $11\cdot 3^{\alpha}+9 = 2\cdot 3^{\beta}$, which forces $\alpha = 2$ and $3^\beta = 54$, a contradiction. Therefore $\mathfrak{\widehat E}_k$ is not a DO polynomial in this case. In the case $k\equiv 13 \pmod {27}$, $k\geq 22$ is equivalent to $k\geq 40$. In this case, consider the $20$th term as given in (\ref{20th}). By similar arguments as in the Case~1 one may prove that the coefficient of $X^{20d}$ is non-zero. Therefore, if $\mathfrak{\widehat E}_k$ is a DO polynomial, then $2d = 3^{\alpha}+1$ and $20d = 3^{\beta}+1$. Combining these equations, we have $10\cdot 3^{\alpha}+9 = 3^{\beta}$, which forces $\alpha = 2$ and $3^\beta = 99$, a contradiction. Therefore $\mathfrak{\widehat E}_k$ is not a DO polynomial in this case.

\textbf{Case 5.} Let $k \equiv 5\pmod{9}$. In this case, if the polynomial $\mathfrak{\widehat E}_{14} = 2a^{10}X^{2d}+a^2X^{6d}+X^{7d}$ is a DO polynomial, then $d = 3^{\alpha}+1$, $6d = 3^t(3^{\beta}+1)$ and $7d = 3^{\gamma}+1$. Since $3^t \mid 6$, $t=1$. Combining the first two equations, we obtain $\alpha =0$, $\beta = 1$ and $d =2$. Now putting these values in third equation, we have $3^{\gamma} = 13$, a contradiction. Therefore, $\mathfrak{\widehat E}_{14}$ is not a DO polynomial. Now, for  $k\geq 23$, consider the $10$th term  
\begin{equation}\label{10th}
\frac{(k-10)(k-11)(k-12)(k-13)\cdots(k-17)(k-18)(k-19)}{10!}a^{k-20}X^{10d}.
\end{equation}
By Lemma~\ref{L2.4M4}, the highest exponent of $3$ that divides $10!$ is $4$. In the numerator of the coefficient of $10$th term, $(k-11), (k-14), (k-17) \equiv 0\pmod{3}$ and $(k-11), (k-17) \not\equiv 0\pmod{9}$. Now if $k \not\equiv 14\pmod{27}$, then the highest exponent of $3$ which divides the numerator is $4$. Hence the coefficient of $X^{10d}$ is non-zero. Thus if $\mathfrak{\widehat E}_k$ is a DO polynomial, then $d = 3^{\alpha}+1$ and $10d = 3^{\beta}+1$. Combining these equations, we get $10\cdot 3^{\alpha}+9 = 3^{\beta}$, which forces $\alpha = 2$ and $3^\beta = 99$, a contradiction. Therefore $\mathfrak{\widehat E}_k$ is not a DO polynomial in this case. In the case $k\equiv 14 \pmod {27}$, $k\geq 23$ is equivalent to $k\geq 41$. Now, consider the $16$th term 
\begin{equation}\label{16th}
\frac{(k-16)(k-17)(k-18)\cdots(k-29)(k-30)(k-31)}{16!}a^{k-32}X^{16d}.
\end{equation}
By way of similar arguments as done in Case~1, the coefficient of $X^{16d}$ is non-zero. Thus, if $\mathfrak{\widehat E}_k$ is a DO polynomial, then $d = 3^{\alpha}+1$ and $16d = 3^{\beta}+1$. Combining these equations, we get $16\cdot 3^{\alpha}+15 = 3^{\beta}$, which forces $\alpha = 1$ and $3^\beta = 63$, a contradiction. Therefore $\mathfrak{\widehat E}_k$ is not a DO polynomial in this case.

\textbf{Case 6.} Let $k \equiv 6\pmod{9}$. In this case, if the polynomial $\mathfrak{\widehat E}_{15} = a^{13}X^{d}+2a^9X^{3d}+aX^{7d}$ is a DO polynomial, then $d = 3^{\alpha}+1$, $3d = 3^t(3^{\beta}+1)$ and $7d = 3^{\gamma}+1$. Since $3^t \mid 3$, $t=1$. Combining these equations, we obtain $\alpha =1$, $\gamma = 3$ and $d =4$. Now, for $k\geq 24$, consider the $10$th term as given in (\ref{10th}). One may follow the similar arguments of Case 5 above to shows that if $k \not\equiv 15\pmod{27}$, the coefficient of $X^{10d}$ is non-zero. Therefore $\mathfrak{\widehat E}_k$ is not a DO polynomial in this case. In the case $k\equiv 15 \pmod {27}$, $k \geq 24$ is equivalent to $k\geq 42$. In this case, consider the $16$th term as given in (\ref{16th}). 
Similar arguments as in the Case~1 show that the coefficient of $X^{16d}$ is non-zero. Therefore $\mathfrak{\widehat E}_k$ is not a DO polynomial in this case.

\textbf{Case 7.} Let $k \equiv 7\pmod{9}$. In this case $k\geq8$ is equivalent to $k\geq 16$. Also, since $k \equiv 1\pmod{3}$, we have $k \not\equiv 0\,\,\textnormal{or}\,\,2\pmod{3}$ and hence the coefficient of $X^{2d}$ in (\ref{E4.1}) is non-zero. Now consider the $8$th term, which is given by
\begin{equation} \label{8th}
\frac{(k-8)(k-9)(k-10)(k-11)(k-12)(k-13)(k-14)(k-15)}{8!}a^{k-16}X^{8d}.
\end{equation}
By following similar arguments as in the Case 1, it is not difficult to prove that the coefficient of $X^{8d}$ is non-zero. Thus, if $\mathfrak{\widehat E}_k$ is a DO polynomial, then $2d = 3^{\alpha}+1$ and $8d = 3^{\beta}+1$. Combining these equations, we have $4\cdot 3^{\alpha}+3 = 3^{\beta}$, which forces $\alpha = 1$ and $3^\beta = 15$, a contradiction. Therefore $\mathfrak{\widehat E}_k$ is not a DO polynomial in this case.
This completes the proof.
\end{proof}

\begin{thm} Let $q$ be a power of an odd prime $p\geq 5$ and $a \in \mathbb{F}_{q}^*$. The polynomial $\mathfrak{\widehat E}_k$ is a DO polynomial over $\mathbb{F}_q$ if and only if one of the following holds.
\begin{enumerate}[{\normalfont (i)}]
\item $k=2,3$ and $d= p^n(p^{\alpha}+1)$, where $\alpha, ~n \geq 0.$
\item $k=7$, $p=5$ and $d= 2p^n$, where $n \geq 0$.
\end{enumerate}
\end{thm}

\begin{proof} It is enough to prove the necessary part. If the polynomials $\mathfrak{\widehat E}_2= -X^d$ and $\mathfrak{\widehat E}_3= -2aX^d$ are DO polynomial, then $d$ is of the form $p^{\alpha}+1$. By Lemma~\ref{L2.1M4}, the polynomials $\mathfrak{\widehat E}_4 = -3a^2X^d+X^{2d}$ and $\mathfrak{\widehat E}_5 = -4a^3X^d+3aX^{2d}$ are not DO polynomials. The polynomial $\mathfrak{\widehat E}_6 = -5a^4X^d+6a^2X^{2d}-X^{3d}$ is a DO polynomial only if $2d = p^{\alpha}+1$ and $3d = p^{\beta}+1$. Combining these equations, we get $3p^{\alpha}+1 = 2p^{\beta}$, which forces $\alpha = 0$, $p^{\beta} =2$, a contradiction. Therefore, $\mathfrak{\widehat E}_6$ is not a DO polynomial. For the polynomial $\mathfrak{\widehat E}_7 = -6a^5X^d+10a^3X^{2d} -4aX^{3d}$, we consider two cases, namely, $p=5$ and $p>5$. For $p=5$, if $\mathfrak{\widehat E}_7 = 4a^5X^d+X^{3d}$ is a DO polynomial, then $d = 5^{\alpha}+1$ and $3d = 5^{\beta}+1$. Combining these equations, we have $3\cdot 5^{\alpha}+2 = 5^{\beta}$, which forces $\alpha = 0$, $\beta =1$ and $d=2$. For  $p>5$, $\mathfrak{\widehat E}_7 = -6a^5X^d+10a^3X^{2d}-4X^{3d}$. Since the coefficients of $X^d$ and $X^{2d}$ are non-zero, Lemma~\ref{L2.1M4} confirms that $\mathfrak{\widehat E}_7$ is not a DO polynomial. For $k\geq8$, we shall consider four cases, namely, $k \not\equiv 1,2,3 \pmod p$, $k\equiv 1 \pmod p$, $k\equiv 2 \pmod p$ and $k \equiv 3 \pmod p$, respectively.

\textbf{Case 1.} Let $k \not\equiv 1,2,3\pmod{p}$. In this case, the coefficients of $X^d$ and $X^{2d}$ in (\ref{E4.1}) are non-zero, therefore $\mathfrak{\widehat E}_k$ is not a DO polynomial by Lemma~\ref{L2.1M4}. 

\textbf{Case 2.} Let $k \equiv 1\pmod{p}$. In this case, we have $(k-2), (k-3), (k-4), (k-5) \not\equiv 0 \pmod p$. Therefore, the coefficients of $X^{2d}$ and $X^{3d}$ in (\ref{E4.1}) are non-zero. Thus, if $\mathfrak{\widehat E}_k$ is a DO polynomial, then $2d = p^{\alpha}+1$ and $3d = p^{\beta}+1$. Combining these equations, we have $3p^{\alpha}+1=2p^{\beta}$, which forces $\alpha = 0$ and $p^{\beta}=2$, a contradiction. Therefore $\mathfrak{\widehat E}_k$ is not a DO polynomial. 

\textbf{Case 3.} Let $k \equiv 2\pmod{p}$. In this case, the coefficient of the first term in (\ref{E4.1}), which contains the monomial $X^d$, is non-zero. Now we consider two cases, namely, $p=5$ and $p>5$. In the case $p=5$, $k\geq 8$ is equivalent to $k\geq 12$. We now show that if $k \not\equiv 7 \pmod{25}$, then the sixth term exists whose coefficient is given by $\frac{(k-6)(k-7)(k-8)(k-9)(k-10)(k-11)}{6!}a^{k-12}.$
Since $k\equiv 2 \pmod 5$, we have $(k-6), (k-8), (k-9), (k-10), (k-11)\not\equiv 0 \pmod 5 $. Also, if $k\not\equiv7 \pmod {25}$, then the highest exponent of $5$ which divides the numerator is 1. By Lemma~\ref{L2.4M4}, the highest exponent of $5$ that divides $6!$  is 1. Therefore the coefficient of $X^{6d}$ is non-zero. Thus, if $\mathfrak{\widehat E}_k$ is a DO polynomial, then $d = 5^{\alpha}+1$ and $6d = 5^{\beta}+1$. Combining these equations, we have $6\cdot 5^{\alpha}+5 = 5^{\beta}$, which forces $\alpha =1$ and $5^{\beta} = 35$, a contradiction. Therefore $\mathfrak{\widehat E}_k$ is not a DO polynomial in this case. Now if $k \equiv 7 \pmod{25}$, then the condition $k\geq 12$ is equivalent to $k \geq 32$. In this case, using the similar arguments, we can show that the coefficient of $X^{8d}$ is non-zero. Thus, if $\mathfrak{\widehat E}_k$ is a DO polynomial, then $d = 5^{\alpha}+1$ and $8d = 5^{\beta}+1$. Combining these equations, we have $8\cdot 5^{\alpha}+7 = 5^{\beta}$, which forces $\alpha =0$ and $5^{\beta} = 15$, a contradiction. Therefore $\mathfrak{\widehat E}_k$ is not a DO polynomial in this case. In the case $p>5$, since $k\equiv 2 \pmod p$, we have $k \not\equiv 1,3,4,5 \pmod p$. Hence the coefficients of $X^d$ and $X^{3d}$ in (\ref{E4.1}) are non-zero, therefore $\mathfrak{\widehat E}_k$ is not DO polynomial by Lemma~\ref{L2.2M4}.

\textbf{Case 4.} Let $k \equiv 3\pmod{p}$. In this case, the first term $(1-k)X^d$ in (\ref{E4.1}) does not vanish. Now we consider two cases, namely, $p=5$ and $p>5$. In the case $p=5$, since $k\equiv 3 \pmod 5$, we have $k \not\equiv 0,1,2,4 \pmod{5}$, and hence the fourth term as given in (\ref{E4th}) does not vanish. Therefore, if $\mathfrak{\widehat E}_k$ is a DO polynomial, then $d=5^i+1$ and $4d=5^j+1$. Combining these equations, we have $4\cdot5^i+3=5^j$, which forces $i=0$ and $5^j=7$, a contradiction. Therefore $\mathfrak{\widehat E}_k$ is not a DO polynomial in this case. In the case $p>5$, since $k\equiv 3 \pmod p$, we have $(k-1), (k-4), (k-5), (k-6), (k-7), (k-8), (k-9) \not\equiv 0 \pmod p$. Therefore, the fourth term as given in (\ref{E4th}) and the fifth term whose coefficient is given by $\frac{(k-5)(k-6)(k-7)(k-8)(k-9)}{5!}a^{k-10},$ do not vanish. Thus, if $\mathfrak{\widehat E}_k$ is a DO polynomial, then $d = p^{\alpha}+1$, $4d = p^{\beta}+1$ and $5d = p^{\gamma}+1$. Combining the first two equations, we have $4p^{\alpha}+3 = p^{\beta}$, which forces $\alpha =0$, $\beta = 1$, $p=7$ and $d =2$. Now putting these values in third equation, we have $7^{\gamma} = 9$, a contradiction. Therefore $\mathfrak{\widehat E}_k$ is not a DO polynomial in this case. This completes the proof.  
\end{proof}

One may recall from \cite[Theorem 3.1]{WY} that RDPs of the second kind and RDPs of the third kind admit the following relationship $$D_{k,2}(a,X) = aD_{k-1,1}(a,X).$$ Thus, it is obvious that $\mathfrak{\widehat F}_k$ is a DO polynomial whenever $\mathfrak{\widehat E}_{k-1}$ is a DO polynomial. Consequently, the classification of DO polynomials from RDPs of the third kind $\mathfrak{\widehat F}_k$ follows immediately. In view of this, we shall consider RDPs of the fourth kind in the next section. 

\section{DO polynomials from RDPs of the fourth kind}\label{S5}

Recall that $D_{k,3}(a,X^d)- D_{k,3}(a,0)$ is denoted by $\mathfrak{\widehat G}_k$, where
\begin{equation}
\begin{split}
\mathfrak{\widehat G}_k=(3-k)a^{k-2}X^d+\frac{(k-3)(k-6)}{2}a^{k-4}X^{2d}- \frac{(k-4)(k-5)(k-9)}{3!}a^{k-6}X^{3d}+ \cdots.  
\end{split}
\end{equation}
Also, from (\ref{E1.2}), it is easy to see that $\mathfrak{\widehat G}_k = \mathfrak{\widehat D}_k \pmod 3$. Therefore, for $p=3$, $\mathfrak{\widehat G}_k$ is a DO polynomial whenever $\mathfrak{\widehat D}_k$ is a DO polynomial and the classification of DO polynomials from $\mathfrak{\widehat D}_k$ has already been discussed in Section \ref{S3}. Therefore, throughout this section, we consider $p\geq 5$. The following theorem gives a complete classification of DO polynomials derived from $\mathfrak{\widehat G}_k$.
\begin{thm} Let $q$ be a power of an odd prime $p\geq5$ and $a \in \mathbb{F}_{q}^*$. The polynomial $\mathfrak{\widehat G}_{k}$  is a DO polynomial over $\mathbb{F}_q$ if and only if one of the following holds.

\begin{enumerate}[{\normalfont (i)}]
\item $k=2$ and $d= p^n(p^{\alpha}+1)$, where $\alpha, ~n \geq 0$.
\item $k=6,11$, $p=5$ and $d= 2p^n$, where $n \geq 0$.
\end{enumerate}
\end{thm}

\begin{proof} It is enough to prove only the necessary part. If the polynomial $\mathfrak{\widehat G}_2= X^d$ is a DO polynomial, then $d = p^{\alpha}+1$. The polynomial $\mathfrak{\widehat G}_{3}$ is the zero polynomial and hence it is not a DO polynomial. The polynomials $\mathfrak{\widehat G}_4 = -a^2X^{d}-X^{2d}$, $\mathfrak{\widehat G}_5=-2a^3X^d-aX^{2d}$ and $\mathfrak{\widehat G}_{7}= -4a^5X^d+2a^3X^{2d}+2aX^{3d}$ are not DO polynomials by Lemma~\ref{L2.1M4}. In the case of the polynomial $\mathfrak{\widehat G}_6=-3a^4X^d+X^{3d}$, we consider two cases, namely, $p=5$ and $p>5$. In the case $p=5$, if $\mathfrak{\widehat G}_6$ is a DO polynomial, then $d=5^i+1$ and $3d=5^j+1$. Combining these equations, we have $3\cdot5^i+2=5^j$, which is true if and only if $i=0$, $j=1$ and $d=2$. When $p>5$, $\mathfrak{\widehat G}_6$ is not a DO polynomial by Lemma~\ref{L2.2M4}. For $k \geq 8$, we consider two cases, namely, $p=5$ and $p>5$.

\textbf{Case 1.} Let $p>5$. Note that when $k \not\equiv 3,6\pmod{p}$, the coefficients of $X^{d}$ and $X^{2d}$ in $\mathfrak{\widehat G}_k$ are non-zero. Therefore, $\mathfrak{\widehat G}_k$ is not a DO polynomial by Lemma~\ref{L2.1M4}. In the case $k \equiv 3 \pmod{p}$, the coefficient of $X^{3d}$ is non-zero and also, the coefficient of $X^{4d}$ in $\mathfrak{\widehat G}_k$, given by  $\frac{(k-5)(k-6)(k-7)(k-12)}{4!}a^{k-8}$ is non-zero. Therefore, $\mathfrak{\widehat G}_k$ is not a DO polynomial by Lemma~\ref{L2.3M4}. When $k \equiv 6 \pmod{p}$, the coefficients of $X^{d}$ and $X^{3d}$ in $\mathfrak{\widehat G}_k$ are non-zero, therefore, $\mathfrak{\widehat G}_k$ is not a DO polynomial by  Lemma~\ref{L2.2M4}.  

\textbf{Case 2.} Let $p=5$. Notice that when $k \not\equiv 1,3\pmod{5}$, the coefficients of $X^{d}$ and $X^{2d}$ in $\mathfrak{\widehat G}_k$ are non-zero, therefore $\mathfrak{\widehat G}_k$ is not a DO polynomial by  Lemma~\ref{L2.1M4}. In the case $k \equiv 3 \pmod{5}$, the coefficients of $X^{3d}$ and $X^{4d}$ in $\mathfrak{\widehat G}_k$ are non-zero, therefore $\mathfrak{\widehat G}_k$ is not a DO polynomial by Lemma~\ref{L2.3M4}. For $k\equiv1 \pmod 5$, if the polynomial $\mathfrak{\widehat G}_{11}=2a^9X^d+a^5X^{3d}+4aX^{5d}$ is a DO polynomial, then $d = 5^{\alpha}+1$, $3d = 5^{\beta}+1$ and $5d = 5^t(5^{\gamma}+1)$. Since $5^t \mid 5$, $t = 1$. Thus, by combining these equations, we obtain $\alpha = 0$, $ \beta =1$ and $d=2$. For $k\geq16$, since $k\equiv 1 \pmod 5$, we have $k\not\equiv 0,2,3,4 \pmod 5$ and hence the coefficient of $X^{3d}$ in $\mathfrak{\widehat G}_k$ is non-zero. Now consider the $6$th term whose coefficient is given by $\frac{(k-7)(k-8)(k-9)(k-10)(k-11)(k-18)}{6!}a^{k-12}.$ By Lemma~\ref{L2.4M4}, the highest exponent of $5$ which divides $6!$ is $1$. Also, if $k\not\equiv 11 \pmod {25}$, then highest exponent of $5$ that divides the numerator of coefficient of $X^{6d}$ is $1$, hence the coefficient of $X^{6d}$ is non-zero. Thus, if $\mathfrak{\widehat G}_k$ is a DO polynomial, then $3d = 5^{\alpha}+1$ and $6d = 5^{\beta}+1$. Combining these equations, we get $2\cdot 5^{\alpha}+1 = 5^{\beta}$, which forces $\alpha = 0$ and $5^{\beta} =3$, a contradiction. Thus $\mathfrak{\widehat G}_k$ is not a DO polynomial in this case. In the case $k\equiv11 \pmod {25}$, consider the $11$th term whose coefficient is given by $\frac{(k-12)(k-13)\cdots(k-19)(k-20)(k-21)(k-33)}{11!}a^{k-22}.$ It is easy to verify that the coefficient of $X^{11d}$ is non-zero. Thus, if $\mathfrak{\widehat G}_k$ is a DO polynomial, then $3d = 5^{\alpha}+1$ and $11d = 5^{\beta}+1$. Combining these equations, we have $11\cdot 5^{\alpha}+8 = 3\cdot5^{\beta}$, which forces $\alpha = 0$ and $3\cdot 5^{\beta} =19$, a contradiction. Thus $\mathfrak{\widehat G}_k$ is not a DO polynomial in this case. 
\end{proof}

\section{DO polynomials from RDPs of the fifth kind}\label{S6}

Here we consider RDPs of the fifth kind. Recall that  $D_{k,4}(a,X^d)- D_{k,4}(a,0)$ is denoted by $\mathfrak{\widehat H}_k$, where
\begin{equation}
\begin{split}
\mathfrak{\widehat H}_k=(4-k)a^{k-2}X^d+\frac{(k-3)(k-8)}{2}a^{k-4}X^{2d}- \frac{(k-4)(k-5)(k-12)}{3!}a^{k-6}X^{3d}+ \cdots.
\end{split}
\end{equation}
It is easy to see from (\ref{E1.2}) that $\mathfrak{\widehat H}_k = \mathfrak{\widehat E}_k \pmod 3$, thus for $p=3$, $\mathfrak{\widehat H}_k$ is a DO polynomial whenever $\mathfrak{\widehat E}_k$ is a DO polynomial. Thus, throughout this section, we take $p\geq 5$. 
\begin{thm} Let $q$ be a power of an odd prime $p \geq5$ and $a \in \mathbb{F}_{q}^*$. The polynomial $\mathfrak{\widehat H}_{k}$ is a DO polynomial over $\mathbb{F}_q$ if and only if one of the following holds.
\begin{enumerate} [{\normalfont (i)}]
\item $k=2,3$ and $d= p^n(p^{\alpha}+1)$, where $\alpha, ~n \geq 0$.
\item $k=4$ and $d= p^n(p^{\alpha}+1)/2$, where $\alpha,~ n \geq 0$.
\end{enumerate}
\end{thm}

\begin{proof} The sufficiency of the theorem is straightforward. It only remains to show the necessity. If the polynomials $\mathfrak{\widehat H}_2 = 2X^d$ and $\mathfrak{\widehat H}_3 = aX^d$ are DO polynomials, then $d=p^{\alpha}+1$. Similarly, if the polynomial $\mathfrak{\widehat H}_4 = -2X^{2d}$ is a DO polynomial, then $d=({p^{\alpha}+1})/2$. In the case of polynomials $\mathfrak{\widehat H}_5 = -a^3X^d-3aX^{2d}$, $\mathfrak{\widehat H}_6 = -2a^4X^d-3a^2X^{2d}+2X^{3d}$ and $\mathfrak{\widehat H}_7 = -3a^5X^{d}-2a^3X^{2d}+aX^{3d}$, the coefficients of $X^{d}$ and $X^{2d}$ are non-zero. Therefore, $\mathfrak{\widehat H}_5$, $\mathfrak{\widehat H}_6$ and $\mathfrak{\widehat H}_7$ are not DO polynomials by Lemma~\ref{L2.1M4}. The polynomial $\mathfrak{\widehat H}_8 = -4a^6X^d+8a^2X^{3d}-2X^{4d}$ is not a DO polynomial by Lemma~ \ref{L2.3M4}. If the polynomial $\mathfrak{\widehat H}_9 = -5a^7X^{d}+3a^5X^{2d}+10a^3X^{3d}-7aX^{4d}$ is a DO polynomial, then $2d = 5^{\alpha}+1$ and $4d = 5^{\beta}+1$. Combining these equations, we get $2\cdot 5^{\alpha}+1 = 5^{\beta}$, which forces $\alpha =0$, $5^{\beta} = 3$, a contradiction. Thus $\mathfrak{\widehat H}_9$ is not a DO polynomial. For $k \geq 10$, we consider two cases, $p=5$ and $p>5$.

\textbf{Case 1.} Let $p =5$. Notice that when $k \not\equiv 3,4\pmod{5}$, the coefficients of $X^{d}$ and $X^{2d}$ in $\mathfrak{\widehat H}_k$ are non-zero, therefore, $\mathfrak{\widehat H}_k$ is not a DO polynomial by  Lemma~\ref{L2.1M4}. In the case $k \equiv 3 \pmod{5}$, the coefficients of $X^{d}$ is clearly non-zero and also, the coefficient of $X^{4d}$ in $\mathfrak{\widehat H}_k$ given by  $\frac{(k-4)(k-5)(k-6)(k-7)}{4!}a^{k-8}$ is non-zero. Thus, if $\mathfrak{\widehat H}_k$ is a DO polynomial, then $d = 5^{\alpha}+1$ and $4d = 5^{\beta}+1$. Combining these equations, we have $4\cdot 5^{\alpha}+3 = 5^{\beta}$, which forces $\alpha = 0$ and $5^{\beta} =7$, a contradiction. Thus $\mathfrak{\widehat H}_k$ is not a DO polynomial in this case. When $k \equiv 4 \pmod{5}$, the coefficient of $X^{2d}$ in $\mathfrak{\widehat H}_k$ is non-zero. Also, if $k \not\equiv 9\pmod{25}$, the coefficient of $X^{5d}$  in $\mathfrak{\widehat H}_k$ given by $ \frac{(k-6)(k-7)(k-8)(k-9)(k-20)}{5!}a^{k-10}$ is non-zero. Thus, if $\mathfrak{\widehat H}_k$ is a DO polynomial, then $2d = 5^{\alpha}+1$ and $5d = 5^t(5^{\beta}+1)$. Since $5^t \mid 5$, $t=1$ and hence, the second equation reduces to $d = 5^{\beta}+1$. Combining these equations, we have $2\cdot 5^{\beta}+1 = 5^{\alpha}$, which forces $\beta = 0$ and $5^{\alpha} =3$, a contradiction. Thus $\mathfrak{\widehat H}_k$ is not a DO polynomial. In the case $k \equiv 9 \pmod{25}$, the condition $k \geq 10$ is equivalent to $k \geq 34$. Now consider the $9$th term whose coefficient is given by $\frac{(k-10)(k-11)(k-12)\cdots (k-16)(k-17)(k-36)}{9!}a^{k-18}.$ Since $k\equiv 9 \pmod {25}$, we have $k \not\equiv 14~(\mbox{mod}~ 25)$. Hence the highest exponent of $5$, which divides the numerator is $1$. By Lemma \ref{L2.4M4}, highest exponent of $5$, which divides $9!$ is $1$. Therefore, the coefficient of $X^{9d}$ is non-zero. Thus, if $\mathfrak{\widehat H}_k$ is a DO polynomial, then $2d = 5^{\alpha}+1$ and $9d = 5^{\beta}+1$. Combining these two equations, we have $9\cdot 5^{\alpha}+7 = 2\cdot 5^{\beta}$, which forces $\alpha = 0$ and $5^{\beta} =8$, a contradiction. Thus $\mathfrak{\widehat H}_k$ is not a DO polynomial. 

\textbf{Case 2.} Let $p>5$. Notice that when $k \not\equiv 3,4,8\pmod{p}$, the coefficients of $X^{d}$ and $X^{2d}$ in $\mathfrak{\widehat H}_k$ are non-zero, therefore $\mathfrak{\widehat H}_k$ is not a DO polynomial by  Lemma~\ref{L2.1M4}. In the case $k \equiv 3,8 \pmod{p}$, the coefficients of $X^{d}$ and $X^{3d}$ in $\mathfrak{\widehat H}_k$ are non-zero, therefore $\mathfrak{\widehat H}_k$ is not a DO polynomial by Lemma~\ref{L2.2M4}. When $k \equiv 4 \pmod{p}$, the coefficients of $X^{2d}$ and $X^{4d}$ in $\mathfrak{\widehat H}_k$ are non-zero. Thus, if $\mathfrak{\widehat H}_k$ is a DO polynomial, $2d = p^{\alpha}+1$ and $4d = p^{\beta}+1$. Combining these equations, we have $2\cdot p^{\alpha}+1 = p^{\beta}$, which forces $\alpha = 0$ and $p^{\beta} =3$, a contradiction. Thus, $\mathfrak{\widehat H}_k$ is not a DO polynomial. This completes the proof.
\end{proof}

\section{The case $m\geq5$}\label{S7}

For $m \geq 5$, we shall classify DO polynomials from the polynomial $\mathfrak{\widehat D}_{k,m}$, where
\begin{equation}
\begin{split}
\mathfrak{\widehat D}_{k,m}=(m-k)a^{k-2}X^d+\frac{(k-3)(k-2m)}{2}a^{k-4}X^{2d}-\frac{(k-4)(k-5)(k-3m)}{3!}a^{k-6}X^{3d}+\cdots. 
\end{split}
\end{equation}
From (\ref{E1.2}), it is straightforward to see that for $p=3$, $\mathfrak{\widehat D}_{k,m} = \mathfrak{\widehat D}_k, \mathfrak{\widehat E}_k,~\mbox{and}~ \mathfrak{\widehat F }_k$, whenever $m\equiv 0, 1~\mbox{and}~ 2 \pmod 3$, respectively. Similarly, for $p\geq5$, $\mathfrak{\widehat D}_{k,m} = \mathfrak{\widehat D}_k, \mathfrak{\widehat E}_k, \mathfrak{\widehat F }_k, \mathfrak{\widehat G}_k ~\mbox{and}~ \mathfrak{\widehat H}_k$, whenever $m\equiv 0, 1, 2, 3~\mbox{and}~ 4 \pmod p$, respectively. Thus the only cases that remain to be considered are $p>5$ and $m \not\equiv 0,1,2,3,4 \pmod{p}$ for which we have the following theorem.
\begin{thm} Let $q$ be a power of an odd prime $p>5$ and $a \in \mathbb{F}_{q}^*$. The polynomial $\mathfrak{\widehat D}_{k,m}$ where $m\not\equiv 0, 1, 2, 3, 4 \pmod p$ is a DO polynomial over $\mathbb{F}_q$ if and only if one of the following holds.
\begin{enumerate} [{\normalfont (i)}]
\item $k=2,3$ and $d= p^n(p^{\alpha}+1)$, where $\alpha, ~n \geq 0$.
\item $k=5$, $m\equiv 5\pmod p$ and $d= p^n(p^{\alpha}+1)/2$, where $\alpha,~ n \geq 0$.
\item $k=5$, $2m\equiv 5\pmod p$ and $d= p^n(p^{\alpha}+1)$, where $\alpha,~n \geq 0$.
\end{enumerate}
\end{thm}
\begin{proof} Only sufficiency of the theorem is required to be proved. If the polynomials $\mathfrak{\widehat D}_{2,m} = (m-2)X^{d}$ and $\mathfrak{\widehat D}_{3,m} = (m-3)aX^{d}$ are DO polynomials, then $d$ is of the form $p^{\alpha}+1$. The polynomial $\mathfrak{\widehat D}_{4,m} = (m-4)a^2X^{d}+(2-m)X^{2d}$ is not a DO polynomial by Lemma~\ref{L2.1M4}. In the case of the polynomial $\mathfrak{\widehat D}_{5,m} = (m-5)a^3X^{d}+(5-2m)aX^{2d}$, we consider three cases, namely, $m \equiv 5\pmod{p}$, $2m \equiv 5\pmod{p}$ and $m, 2m \not\equiv 5\pmod{p}$. In the case $m \equiv 5\pmod{p}$, if $\mathfrak{\widehat D}_{5,m} =-5aX^{2d}$ is a DO polynomial, then $d$ is of the form $(p^{\alpha}+1)/2$. When $2m \equiv 5\pmod{p}$ and if $\mathfrak{\widehat D}_{5,m} =(m-5)a^3X^{d}$ is a DO polynomial, then $d$ is of the form $p^{\alpha}+1$. In the case $m, 2m \not\equiv 5\pmod{p}$, $\mathfrak{\widehat D}_{5,m} = (m-5)a^3X^{d}+(5-2m)aX^{2d}$ is not a DO polynomial by Lemma \ref{L2.1M4}. For $k \geq 6$, we consider four cases, namely, $k\not\equiv3, m, 2m \pmod p$, $k\equiv3 \pmod p$, $k \equiv m \pmod p$ and $k\equiv 2m \pmod p$.

\textbf{Case 1.} Let $k \not\equiv 3,m,2m\pmod{p}$. In this case, the coefficients of $X^{d}$ and $X^{2d}$ in $\mathfrak{\widehat D}_{k,m}$ are non-zero, and therefore $\mathfrak{\widehat D}_{k,m}$ is not a DO polynomial by Lemma~\ref{L2.1M4}.   

\textbf{Case 2.} Let $k\equiv3\pmod{p}$. In this case, we have $k\not\equiv 4,5 \pmod{p}$. Also, note that $k \not\equiv m\pmod{p}$, otherwise $m\equiv 3\pmod{p}$. Similarly, $k \not\equiv 3m\pmod{p}$, otherwise $m\equiv 1\pmod{p}$. Therefore, the coefficients of $X^{d}$ and $X^{3d}$ in $\mathfrak{\widehat D}_{k,m}$ are non-zero and hence $\mathfrak{\widehat D}_{k,m}$ is not a DO polynomial by  Lemma~\ref{L2.2M4}. 

\textbf{Case 3.} Let $k \equiv m\pmod{p}$. Notice that $k \not\equiv 3,4\pmod{p}$. Also, note that $k \not\equiv 2m, 3m\pmod{p}$, otherwise $m\equiv 0\pmod{p}$. Therefore, the coefficient of $X^{2d}$ in $\mathfrak{\widehat D}_{k,m}$ is non-zero. Also, when $k\not\equiv5\pmod{p}$, the coefficient of $X^{3d}$ in $\mathfrak{\widehat D}_{k,m}$ is non-zero. Thus, if $\mathfrak{\widehat D}_{k,m}$ is a DO polynomial, then $2d = p^{\alpha}+1$ and $3d = p^{\beta}+1$. Combining these equations, we get $3p^{\alpha}+1 = 2p^{\beta}$, which forces $\alpha = 0$ and $p^{\beta}=2$, a contradiction. Therefore, $\mathfrak{\widehat D}_{k,m}$ is not a DO polynomial in this case. In the case $k\equiv5\pmod{p}$, consider the coefficient of the fifth term $\frac{(k-6)(k-7)(k-8)(k-9)(k-5m)}{5!}a^{k-10}.$ Since $k \equiv 5 \pmod p$, we have $(k-6), (k-7), (k-8), (k-9)\not\equiv 0 \pmod{p}$. Also, note that $k\not\equiv5m\pmod{p}$, otherwise $m\equiv1 \pmod p$. Therefore, the coefficients of $X^{2d}$ and $X^{5d}$ in $\mathfrak{\widehat D}_{k,m}$ are non-zero. Thus, if $\mathfrak{\widehat D}_{k,m}$ is a DO polynomial, then $2d = p^{\alpha}+1$ and $5d = p^{\beta}+1$. Combining these equations, we get $5p^{\alpha}+3 = 2p^{\beta}$, which forces $\alpha = 0$ and $p^{\beta}=4$, a contradiction. Therefore, $\mathfrak{\widehat D}_{k,m}$ is not a DO polynomial.

\textbf{Case 4.} Let $k \equiv 2m\pmod{p}$. Notice that $k\not\equiv m\pmod{p}$, otherwise $m\equiv 0\pmod{p}$. Therefore the coefficient of $X^d$ in $\mathfrak{\widehat D}_{k,m}$ is non-zero. Also note that $k\not\equiv 3m\pmod{p}$ and $k\not\equiv4\pmod{p}$, otherwise $m\equiv 0\pmod{p}$ and $m\equiv 2\pmod{p}$, respectively. When $k \not \equiv 5 \pmod p$, the coefficient of $X^{3d}$ in $\mathfrak{\widehat D}_{k,m}$ is non-zero and hence $\mathfrak{\widehat D}_{k,m}$ is not a DO polynomial by  Lemma~\ref{L2.2M4}. In the case $k \equiv 5\pmod{p}$, the condition $k\geq6$ is equivalent to $k\geq 13$. Now consider the fifth term again. By similar arguments as done in Case~3 above, it is easy to see that the coefficient of $X^{5d}$ is non-zero. Thus, if $\mathfrak{\widehat D}_{k,m}$ is a DO polynomial, then $d = p^{\alpha}+1$ and $5d = p^{\beta}+1$. Combining these equations, we get $5p^{\alpha}+4 = p^{\beta}$, which forces $\alpha = 0$ and $p^{\beta}=9$, a contradiction. Therefore $\mathfrak{\widehat D}_{k,m}$ is not a DO polynomial. This completes the proof.
\end{proof}

\section{Discussion on planarity}\label{S8}
We consider the planarity of DO polynomials obtained from RDPs of the $(m+1)$-th kind as listed in the Appendix \ref{AA}. First, we shall discuss the tools and techniques that are needed to understand the planarity of DO polynomials. These tools and techniques are similar to the ones used in \cite{CM}. Recall that a polynomial function $f:  \mathbb{F}_q \rightarrow \mathbb{F}_q$ is said to be planar if the difference function $\Delta_{f}(X, \epsilon)=f(X+\epsilon)-f(X)-f(\epsilon)$ permutes the elements of $\mathbb{F}_q$ for each $\epsilon \in \mathbb{F}_q^*$. If $f$ happens to be a DO polynomial, the difference function $\Delta_{f}(X, \epsilon)$ for each $\epsilon \in \mathbb{F}_q^*$, belongs to another well-known class of polynomials called linearized polynomials. Therefore, a DO polynomial $f$ is planar if and only if the linearized polynomial $\Delta_{f}(X, \epsilon)$ is a permutation polynomial for each $\epsilon \in \mathbb{F}_q^*$. The permutation behaviour of linearized polynomial is well-known. In fact, \cite[Theorem 7.9]{BOOK} tells us that a linearized polynomial is a permutation polynomial over $\mathbb{F}_{q}$ if and only if its only root in $\mathbb{F}_q$ is $0$. Therefore, in order to show that a DO polynomial $f$ is not planar, it is sufficient to show that the difference function $\Delta_{f}(X,Y)= f(X+Y)-f(X)-f(Y)$ has a root in $\mathbb{F}_{q}^* \times \mathbb{F}_{q}^*$.

We recall that a DO polynomial function $f$ from $\mathbb{F}_q$ to itself is called 2-to-1 function if the cardinality of the image set on $\mathbb{F}_q^*$ is $(q-1)/2.$ Qiu et al. \cite{WEA} showed that the size of the image set on $\mathbb{F}_q^*$ of a planar polynomial $f$ over $\mathbb{F}_q$ must be at least $(q-1)/2.$ For a DO polynomial $f$, Weng and Zeng  \cite[Theorem 2.3]{WZ} gave the following necessary and sufficient condition for $f$ to be planar. 

\begin{lem} \label{l9.1}
Let $f$ be a DO polynomial over $\mathbb{F}_q$. Then $f$ is planar if and only if $f$ is 2-to-1.
\end{lem}
Lemma \ref{l9.1} has further consequences. First, if a DO polynomial $f$ has a root $z\in \mathbb{F}_q^*$, then $-z$ is also a root of $f$. Therefore, the cardinality of image set of $f$ on $\mathbb{F}_q^*$ is strictly less than $(q-1)/2$ and hence, in such a case, $f$ is not planar. 

For the second consequence, we begin with an easy observation that if $f(X)$ is a DO polynomial, then so is $f(X^{p^n})$. We know that $X^{p^n}$ is a linearized permutation polynomial over $\mathbb{F}_{p^e}$. Therefore, the cardinality of the image set of $f(X)$ and  $f(X^{p^n})$ on $\mathbb{F}_q^*$ is same. Hence if $f(X)$ is planar, then $f(X^{p^n})$ is also planar. Therefore in such situations, it would be sufficient to consider the planarity of $f(X)$. 

Another important tool that we would require to study the planarity of DO polynomials is the following version of Weil bound as stated in~\cite[Lemma 2.4]{CG}. 
\begin{lem} \label{weil bound}
Let $f(X,Y)$ be an absolutely irreducible polynomial in $\mathbb{F}_q[X,Y]$.  Then the number $N_f$ of $(u,v)\in \mathbb{F}_q \times \mathbb F_q$ with $f(u,v)=0$ satisfies \[N_f \geq q-(d-1)(d-2)\sqrt{q}-d-1,\] where $d$ is the total degree of $f$.
\end{lem} 

We now describe the strategy for using the Weil bound to determine the planarity of certain DO polynomials. Let $f$ be a DO polynomial over $\mathbb{F}_q$ and consider the difference function $\Delta_{f}(X, Y)= f(X+Y)-f(X)-f(Y)$. If this difference function has an absolutely irreducible factor, say $h(X,Y)$, of total degree $d_h$, then Lemma~\ref{weil bound} gives a lower bound for the cardinality $N_h$ of all the points $(u,v) \in \mathbb{F}_q \times \mathbb{F}_q$ such that $h(u,v)=0$. If the degree of the absolutely irreducible factor $h(X,Y)$ is not too large and $q$ is large enough, then we have many $\mathbb F_q$-rational points on the affine algebraic curve defined by $h(X,Y)=0$. Moreover, if $N_h$ is strictly larger than the number of solutions to $h(X,Y)=0$ with either $X=0$ or $Y=0$, then Lemma~\ref{weil bound} yields the existence of a point $(u,v)$ in $\mathbb{F}_q^* \times \mathbb{F}_q^*$ such that $h(u,v)=0$ and hence, for such a point, we have $\Delta_{f}(u, v)= 0$, i.e, $\Delta_{f}(X, Y)$ has a root in $\mathbb{F}_{q}^* \times \mathbb{F}_{q}^*$. Thus, in order to show that $f$ is not exceptional planar (i.e., planar over infinitely many extensions of $\mathbb{F}_q$), it is sufficient to show that the difference function of $f$ contains an absolutely irreducible component with a solution in $\mathbb{F}_q^* \times \mathbb{F}_q^*$.

It is straightforward to see that for $b \in \mathbb{F}_q^*$, RDPs of the $(m+1)$-th kind admit the following relationship 
\begin{equation} \label{rdpr}
\displaystyle b^{kd}D_{k,m}(a,X^d) = D_{k,m}(ab^d, (Xb^2)^{d}).
\end{equation}
In view of (\ref{rdpr}), and due to the fact that the planarity property of a function $f$ remains invariant under linear transformations (i.e. if $f(X)$ is planar so is $\alpha f(\lambda X + \mu)+ \beta$ with $\alpha, \lambda \neq 0$), we have the following lemma.
\begin{lem} \label{PE}
Let $D_{k,m}(a, X)$ be the $k$-th RDP of the $(m+1)$-th kind. Then $D_{k,m}(a, X^d)$ is planar equivalent over $\mathbb F_q$ to $D_{k,m}(ab^d, X^{d})$ for any $b \in \mathbb{F}_q^*$. 
\end{lem}

Over the algebraic closure ${\overline{\mathbb F}_q}$ of $\mathbb F_q$, we derive a useful consequence of Lemma \ref{PE}. Note that one may always choose $b \in \overline{\mathbb{F}}_{q}$ that satisfies the equation $aX^d = 1$. In this way the factorizations of $\Delta_{D_{k,m}(a, X^d)}$ and $\Delta_{D_{k,m}(1,X^d)}$ over ${\overline{\mathbb F}_q}$ are linearly related. As a consequence, the absolutely irreducible factors of $\Delta_{D_{k,m}(a,X^d)}$ are of the same form for all non-zero $a$. Thus, without loss of generality, one may always take $a=1$, while checking the absolute irreducibility of certain polynomials. 

Now we consider the planarity of the DO polynomials listed in the Appendix \ref{AA} in three different cases.

\textbf{Case 1.} Let $p=3$. The planarity of monomials $\mathfrak{\widehat D}_2 = X^{3^{\alpha}+1}$, $\mathfrak{\widehat E}_2 = 2X^{3^{\alpha}+1}$, $\mathfrak{\widehat E}_3 = aX^{3^{\alpha}+1}$, $\mathfrak{\widehat E}_4 = X^{3^{\alpha}+1}$, and $\mathfrak{\widehat E}_5 = 2a^3 X^{3^{\alpha}+1}$ is well-known by  \cite[Theorem 3.3]{CMLB} and these monomials are planar over $\mathbb{F}_{3^e}$ if and only if $e/(\alpha,e)$ is odd. It is easy to see that $X=a$ is a root of the polynomials $\mathfrak{\widehat D}_5 = 2aX^{4}+a^3X^2$, $\mathfrak{\widehat E}_7 = 2aX^{6}+a^3X^4$, $\mathfrak {\widehat {D}}_7 =  2aX^{6}+2a^3X^4+ 2a^5X^2$, $\mathfrak{\widehat E}_{13} = aX^{12}+a^3X^{10}+a^9X^4$ and $\mathfrak{\widehat E}_{19} = 2aX^{18}+2a^9X^{10}+a^{13}X^6+a^{15}X^4$. Therefore, these DO polynomials are not planar. Now we consider the planarity of the rest of the DO polynomials one by one.

(\romannumeral 1)~~  In the case of binomial $f(X) = \mathfrak{\widehat D}_4 = 2X^{4}+2a^2X^2$, consider the difference function $\Delta_{f}(X,Y) =  f(X+Y)-f(X)-f(Y) = XYB(X,Y),$ where $B(X,Y)=X^2+Y^2-a^2$, which is simply an irreducible conic since $a$ is non-zero. Therefore, by Lemma \ref{weil bound}, the number $N_B$ of $(u,v)\in \mathbb{F}_q\times \mathbb{F}_q$ with $B(u,v)=0$ is greater than or equal to $q-3$. Note that we can obtain at most $4$ solutions $(u,v)$ to $B(X,Y)=0$ by putting either $X=0$ or $Y=0$. Therefore, when $q-3>4$, there must exist a root $(u,v)\in \mathbb{F}_{q}^* \times \mathbb{F}_{q}^* $ of $B(X,Y)$. Therefore, $\mathfrak{\widehat D}_4$ is not planar when $q>7$, i.e., $e\geq 2$. For $e=1$, $\mathfrak{\widehat D}_4\equiv X^2~(\mbox{mod}~X^3-X)$ which is clearly a planar function.

(\romannumeral 2)~~  The DO binomial $\mathfrak{\widehat E}_6 = 2X^{3(3^{\alpha}+1)}+a^4X^{3^{\alpha}+1}$ can be written as composition of a linearized polynomial and a monomial as $(2X^3+a^4X) \circ X^{3^{\alpha}+1}$. Now from \cite[Theorem 2.3]{CMLB}, $\mathfrak{\widehat E}_6$ is planar if and only if $2X^3+a^4X$ is a permutation polynomial and $X^{3^{\alpha}+1}$ is planar. Now, since $X=a^2$ is a root of the linearized polynomial $2X^3+a^4X$, $2X^3+a^4X$ is not a permutation polynomial. Hence, $\mathfrak{\widehat E}_6$ is not planar.

(\romannumeral 3)~~  In the case of the DO polynomial $f(X)= \mathfrak{\widehat E}_{10}  = 2X^{10}+a^4X^6+a^6X^4$, consider the difference function $\Delta_{f} (X,Y)=XY~h(X,Y),$ where $h(X,Y)=2(X^8+Y^8)-a^4X^2Y^2 +a^6(X^2+Y^2).$ The Magma algebra package \cite{magma} reveals that $h(X,Y)$ is absolutely irreducible. Therefore, by Lemma \ref{weil bound}, the number $N_h$ of solutions $(u,v) \in \mathbb{F}_q \times \mathbb{F}_q$ of $h(X,Y)=0$ satisfies $N_h \geq q-42\sqrt{q}-9.$ Now $h(X,0)= 2X^8+a^6X^2 = X^2(a+X)^3(a-X)^3$ have in total $8$ solutions in $\mathbb{F}_q$. Similarly, $h(0,Y)=2Y^8a+a^6Y^2=Y^2(a-Y)^3(a+Y)^3$ have in total $8$ solutions in $\mathbb{F}_q$. Therefore, in total $16$ solutions can be obtained either by putting $X=0$ or $Y=0$. Now if $q-42\sqrt{q}-9>16$, i.e., $q-42\sqrt{q}-25>0$ then $h(X,Y)$ possesses a solution $(u,v)\in \mathbb{F}_q^* \times \mathbb{F}_q^*$. This is true for $e\geq 7$, therefore, for $e\geq 7$, $\mathfrak{\widehat E}_{10}$ is not planar. For $e=1$ $\mathfrak{\widehat E}_{10}=X^2 \pmod {X^3-X}$, which is clearly a planar polynomial. Computations show that for $2 \leq e \leq 6$, the cardinality of the image set of $\mathfrak{\widehat E}_{10}$ on $\mathbb{F}_q^*$ is strictly less than $(3^e-1)/2$. Therefore, by Lemma \ref{l9.1}, $\mathfrak{\widehat E}_{10}$ is not planar in these cases.

(\romannumeral 4)~~  Consider the DO polynomial $f(X)=\mathfrak{\widehat E}_{15} =aX^{28}+2a^9X^{12}+a^{13}X^4 = a(X^7+2a^8X^3+a^{12}X) \circ X^4.$ 
This polynomial is never planar over $\mathbb{F}_{3^e}$ when $e$ is even. Since in this case $4\mid (q-1)$, the cardinality of image set of $f(X)$ on $\mathbb{F}_{q}^*$ is at most $(q-1)/4$ and thus, by Lemma~\ref{l9.1}, $f(X)$ is not planar. When $e$ is odd, we consider the difference function $\Delta_{f}(X,Y) = aXY(X^2+Y^2)~h(X,Y),$ where
$$ h(X,Y) = a^{12}+\sum_{i=0}^{12}(-1)^iX^{24-2i}Y^{2i}+\sum_{i=1}^{3}(-1)^ia^8X^{8-2i}Y^{2i}.$$ Again, the Magma algebra package \cite{magma} shows that the polynomial $h'(X,Y)$ obtained from $h(X,Y)$ by putting $a=1$, is absolutely irreducible. Therefore, by Lemma \ref{weil bound}, the number $N_{h'}$ of solutions $(u,v) \in \mathbb{F}_q \times \mathbb{F}_q$ of $h'(X,Y)=0$ satisfies $N_{h'} \geq q-506\sqrt{q}-25$. Also, $h'(X,0) = X^{24}+1$ and this has no root in odd degree extensions of $\mathbb{F}_3$. Similarly, $h'(0,Y)= Y^{24}+1$ has no root in odd degree extensions of $\mathbb{F}_3$. Therefore, there is no solution to $h'(X,Y)=0$ corresponding to $XY=0$. If $ q-506\sqrt{q}-25 > 0$, then $h'(X,Y)$ has a root $(u,v)\in \mathbb{F}_q^* \times \mathbb{F}_q^*.$ This holds true for all $e \geq 12$. Therefore, $\mathfrak{\widehat E}_{15}$ is not planar over $\mathbb F_{3^e}$ for $e \geq 12$. In the case $e=1$, the polynomial $f(X) = \mathfrak{\widehat E}_{15} = aX^2 \pmod{X^3-X}$ which is clearly a planar function. Computations show that for $e = 5, 7, 9, 11$, the cardinality of the image set of $\mathfrak{\widehat E}_{15}$ on $\mathbb F_{3^e}^{*}$ is strictly less than $(3^e-1)/2$, therefore, $\mathfrak{\widehat E}_{15}$ is not planar in these cases. In the case $e=3$, $\mathfrak{\widehat E}_{15}$ is planar for every choice of $a\in \mathbb{F}_{27}^{*}$.

\textbf{Case 2.} Let $p=5$. The planarity of DO monomials $\mathfrak{\widehat D}_2 = 3X^{5^{\alpha}+1}$, $\mathfrak{\widehat D}_3 = 2aX^{5^{\alpha}+1}$, $\mathfrak{\widehat E}_2 = 4X^{5^{\alpha}+1}$, $\mathfrak{\widehat E}_3 = 3aX^{5^{\alpha}+1}$, $\mathfrak{\widehat G}_2 = X^{5^{\alpha}+1}$, $\mathfrak{\widehat H}_2 = 2X^{5^{\alpha}+1}$, $\mathfrak{\widehat H}_3 = aX^{5^{\alpha}+1}$, and $\mathfrak{\widehat H}_4 = 3X^{5^{\alpha}+1}$ is well-known by \cite[Theorem 3.3]{CMLB} and these monomials are planar over $\mathbb{F}_{5^e}$ whenever $e/(\alpha, e)$ is odd. It is straightforward to see that $X=a$ is a root of the DO binomial $\mathfrak{\widehat E}_7 = 4a^5X^2+aX^6$ and hence, it is not planar. Now we consider the planarity of the rest of the DO polynomials one by one.

(\romannumeral 1)~~  For the DO binomial $f(X) = \mathfrak{\widehat G}_6 = 2a^4X^2+X^6$, consider the difference function $\Delta_{f}(X,Y) = XYB(X,Y),$ where $B(X,Y)=X^4+Y^4-a^4$. It is easy to see that $Y-a\mid Y^4-a^4$ and $Y^4-a^4$ has no repeated roots. Therefore, by Eisenstein's criterion, $B(X,Y)$ is absolutely irreducible. Thus, by Lemma \ref{weil bound}, the number of solutions $(u,v)\in \mathbb{F}_q \times \mathbb{F}_q$ of $B(X,Y)=0$ satisfies $ N_B\geq q-6\sqrt{q}-5.$ Now, at most $8$ roots of $B(X,Y)$ can be obtained by putting either $X=0$ or $Y=0$. Therefore, if $q-6\sqrt{q}-5 > 8$, $B(X,Y)$ will have a solution $(u,v)\in \mathbb{F}_{q}^*\times \mathbb{F}_{q}^*$, which holds for all $e\geq3$. Therefore, $\mathfrak{\widehat G}_6$ is not planar over $\mathbb F_{5^e}$ for $e\geq3$. When $e=1$, $f(X)=3X^2(\mbox{mod}~(X^5-X))$ which is clearly a planar function. For $e=2$, the number of solutions of the equation $X^4+Y^4=a^4$ in $\mathbb{F}_{5^2} \times \mathbb{F}_{5^2}$ is 40, which is greater than 16. Therefore, $\mathfrak{\widehat G}_6$ is not planar in this case.

(\romannumeral 2)~~ In the case of the DO trinomial $f(X) = \mathfrak{\widehat G}_{11} = -aX^{10}+a^5X^6+2a^9X^2,$ consider the difference function $\Delta_{f} (X,Y)=XY~h(X,Y),$ where $h(X,Y)=3aX^4Y^4+a^5X^4 + a^5Y^4 +4a^9.$ The Magma algebra package \cite{magma} shows that $h(X,Y)$ is absolute irreducible. Therefore, by Lemma \ref{weil bound}, the number $N_h$ of solutions $(u,v) \in \mathbb{F}_q \times \mathbb{F}_q$ of $h(X,Y)=0$ satisfies $N_h \geq q-42\sqrt{q}-9$. Now, $h(X,0)= X^4-a^4 = 0$ can have at most $4$ solutions. Similarly, $h(0,Y)= Y^4-a^4 = 0$ can have at most $4$ solutions. Therefore, at most $8$ solutions can be obtained by putting either $X=0$ or $Y=0$. Now, if $q-42\sqrt{q}-9>8,$ i.e., $q-42\sqrt{q}-17>0$ then $h(X,Y)$ will have a solution $(u,v)\in \mathbb{F}_q^* \times \mathbb{F}_q^*$. This is true for $e\geq 5$, therefore, for $e\geq 5$, $\mathfrak{\widehat G}_{11}$ is not planar. For $e=1$, $\mathfrak{\widehat G}_{11}= 2aX^2$ is clearly a planar function. For $e=2,4$, computations show that the cardinality of the image set of $\mathfrak{\widehat G}_{11}$ on $\mathbb F_{5^e}^{*}$ is strictly less than $(5^e-1)/2$. Therefore, $\mathfrak{\widehat G}_{11}$ is not planar in these cases. For $e=3$, computations show that $\mathfrak{\widehat G}_{11}$ is planar for every choice of $a \in \mathbb{F}_{125}^*$.

\textbf{Case 3.} Let $p > 5$. In this case, the only DO polynomials we are getting are the monomials of the form $bX^{p^{\alpha}+1}$ where $b \in \mathbb{F}_q^*$  and by \cite[Theorem 3.3]{CMLB}, these monomials are planar over $\mathbb{F}_{p^e}$ whenever $e/(\alpha,e)$ is odd.

In view of the foregoing discussion, the following theorem gives the list of planar DO polynomials arising from RDPs of arbitrary kind.

\begin{thm}  
Let $\displaystyle \mathfrak{\widehat D}_{k,m} = \sum_{i=1}^{\lfloor\frac k2\rfloor}\frac{k-mi}{k-i}\dbinom{k-i}{i}(-X^d)^{i}a^{k-2i}$ as defined in the Introduction. Then the following are the only planar DO polynomials arising from $\mathfrak{\widehat D}_{k,m}.$
\begin{enumerate} [{\normalfont (i)}]
 \item $X^2$ over $\mathbb{F}_{p^e}$.
 \item $X^{p^\alpha +1}$ over $\mathbb{F}_{p^e}$ with $\frac{e}{(\alpha,e)}$ odd.
 \item $2a^9X^{12}+a^{13}X^4+aX^2$ over $\mathbb{F}_{27}$ with $a\in \mathbb{F}_{27}^*$.
 \item $-aX^{10}+a^5X^6+2a^9X^2$ over $\mathbb{F}_{125}$ with $a\in \mathbb{F}_{125}^*$.
\end{enumerate}
\end{thm}
\section*{Acknowledgments}
The authors would like to express their sincere appreciation for the reviewers’ careful
reading, beneficial comments and suggestions, and to the editors for the prompt handling of our paper. The research of Sartaj Ul Hasan is partially supported by Start-up Research Grant SRG/2019/000295 from the Science and Engineering Research Board, Government of India

\appendix 
\section{The complete list of DO polynomials} \label{AA}
Here, we present the complete list of DO polynomials obtained from polynomial $\mathfrak{\widehat D}_{k,m}$ over a finite field of odd characteristic.
 
 \begin{enumerate}
  \item The case $p=3$.
 \begin{enumerate}
  \item When $m \equiv 0\pmod{3}$
  \begin{enumerate}
 \item $k=2\cdot3^{\ell}$, $X^{3^{n+\ell}(3^{\alpha}+1)}$ for nonnegative integers $\alpha$, $n$ and $\ell$.
 \item $k=4\cdot3^\ell$, $2a^2X^{2\cdot3^{n+\ell}}+2X^{4\cdot3^{n+\ell}}$ for nonnegative integers $n$ and $\ell$.
 \item $k=5\cdot3^\ell$, $a^3X^{2\cdot3^{n+\ell}}+2aX^{4\cdot3^{n+\ell}}$ for nonnegative integers $n$ and $\ell$.
 \item $k=7\cdot3^\ell$, $2a^5X^{2\cdot3^{n+\ell}}+2a^3X^{4\cdot3^{n+\ell}}+2aX^{2\cdot3^{n+\ell+1}}$ for nonnegative integers $n$ and $\ell$.
  \end{enumerate}
  \item When $m \equiv 1\pmod{3}$
  \begin{enumerate}
   \item $k=2$, $2X^{3^n(3^{\alpha}+1)}$ for nonnegative integers $\alpha$ and $n$.
   \item $k=3$, $aX^{3^{n}(3^{\alpha}+1)}$ for nonnegative integers $\alpha$ and $n$.
   \item $k=4$, $X^{3^{n}(3^{\alpha}+1)}$ for nonnegative integers $\alpha$ and $n$.
   \item $k=5$, $2a^3X^{3^n(3^{\alpha}+1)}$ for nonnegative integers $\alpha$ and $n$.
   \item $k=6$, $a^4X^{3^n(3^{\alpha}+1)}+2X^{3^{n+1}(3^{\alpha}+1)}$ for nonnegative integers $\alpha$ and $n$. 
   \item $k=7$, $a^3X^{4\cdot3^{n}}+2aX^{2\cdot3^{n+1}}$ for nonnegative integer $n$.
   \item $k=10$, $a^6X^{4\cdot3^{n}}+a^4X^{2\cdot3^{n+1}}+2X^{10\cdot3^{n}}$ for nonnegative integer $n$.
   \item $k=13$, $a^9X^{4\cdot3^{n}}+a^3X^{10\cdot3^{n}}+aX^{4\cdot3^{n+1}}$ for nonnegative integer $n$.
   \item $k=15$, $a^{13}X^{4\cdot3^{n}}+2a^9X^{4\cdot3^{n+1}}+aX^{28\cdot3^{n}}$ for nonnegative integer $n$.
   \item $k=19$, $a^{15}X^{4\cdot3^{n}}+a^{13}X^{2\cdot3^{n+1}}+2a^9X^{10\cdot3^{n}}+2aX^{2\cdot3^{n+2}}$ for nonnegative integer $n$.
  \end{enumerate}
  \item When $m \equiv 2\pmod{3}$
   \begin{enumerate}
   \item $k=3$, $2aX^{3^n(3^{\alpha}+1)}$ for nonnegative integers $\alpha$ and $n$.
   \item $k=4$, $a^2X^{3^{n}(3^{\alpha}+1)}$ for nonnegative integers $\alpha$ and $n$.
   \item $k=5$, $aX^{3^{n}(3^{\alpha}+1)}$ for nonnegative integers $\alpha$ and $n$.
   \item $k=6$, $2a^4X^{3^n(3^{\alpha}+1)}$ for nonnegative integers $\alpha$ and $n$.
   \item $k=7$, $a^5X^{3^n(3^{\alpha}+1)}+2aX^{3^{n+1}(3^{\alpha}+1)}$ for nonnegative integers $\alpha$ and $n$. 
   \item $k=8$, $a^4X^{4\cdot3^{n}}+2a^2X^{2\cdot3^{n+1}}$ for nonnegative integer $n$.
   \item $k=11$ $a^7X^{4\cdot3^{n}}+a^5X^{2\cdot3^{n+1}}+2aX^{10\cdot3^{n}}$ for nonnegative integer $n$.
   \item $k=14$, $a^{10}X^{4\cdot3^{n}}+a^4X^{10\cdot3^{n}}+a^2X^{4\cdot3^{n+1}}$ for nonnegative integer $n$.
   \item $k=16$, $a^{14}X^{4\cdot3^{n}}+2a^{10}X^{4\cdot3^{n+1}}+a^2X^{28\cdot3^{n}}$ for nonnegative integer $n$.
   \item $k=20$, $a^{16}X^{4\cdot3^{n}}+a^{14}X^{2\cdot3^{n+1}}+2a^{10}X^{10\cdot3^{n}}+2a^2X^{2\cdot3^{n+2}}$ for nonnegative integer $n$.
  \end{enumerate}
 \end{enumerate}
 \item The case $p=5$.
  \begin{enumerate}
  \item When $m \equiv 0\pmod{5}$
  \begin{enumerate}
   \item $k=2\cdot5^\ell$, $3X^{5^{n+\ell}(5^{\alpha}+1)}$ for non negative integers $\alpha$, $n$ and $\ell$.
   \item $k=3\cdot5^\ell$, $2aX^{5^{n+\ell}(5^{\alpha}+1)}$ for non negative integers $\alpha$, $n$ and $\ell$.
  \end{enumerate}
  \item When $m \equiv 1\pmod{5}$
  \begin{enumerate}
   \item $k=2$, $4X^{5^n(5^{\alpha}+1)}$ for nonnegative integers $\alpha$ and $n$.
   \item $k=3$, $3aX^{5^{n}(5^{\alpha}+1)}$ for nonnegative integers $\alpha$ and $n$.
   \item $k=7$, $4a^5X^{2\cdot5^{n}}+aX^{6\cdot5^{n}}$ for nonnegative integer $n$.
   \end{enumerate}
  \item When $m \equiv 2\pmod{5}$
  \begin{enumerate}
   \item $k=3$, $4aX^{5^n(5^{\alpha}+1)}$ for nonnegative integers $\alpha$ and $n$.
   \item $k=4$, $3a^2X^{5^{n}(5^{\alpha}+1)}$ for nonnegative integer $n$.
   \item $k=8$, $4a^6X^{2\cdot5^{n}}+a^2X^{6\cdot5^{n}}$ for nonnegative integer $n$.
  \end{enumerate}
  \item When $m \equiv 3\pmod{5}$
  \begin{enumerate}
  \item $k=2$, $2X^{5^n(5^{\alpha}+1)}$ for nonnegative integers $\alpha$ and $n$.
  \item $k=6$, $2a^4X^{2\cdot5^{n}}+X^{6\cdot5^{n}}$ for nonnegative integer $n$.
  \item $k=11$, $2a^9X^{2\cdot5^{n}}+a^5X^{6\cdot5^{n}}+4aX^{2\cdot5^{n+1}}$ for nonnegative integer $n$.
  \end{enumerate}
  \item When $m \equiv 4\pmod{5}$
  \begin{enumerate}
   \item $k=2$, $2X^{5^n(5^{\alpha}+1)}$ for nonnegative integers $\alpha$ and $n$.
   \item $k=3$, $aX^{5^{n}(5^{\alpha}+1)}$ for nonnegative integer $n$.
   \item $k=4$, $3X^{5^{n}(5^{\alpha}+1)}$ for nonnegative integer $n$.
  \end{enumerate}
 \end{enumerate}
 \item The case $p>5$.\\
  In this case, we are getting DO polynomials of the form $bX^{p^{\alpha}+1}$ where $b\in \mathbb{F}_q^*$.
  
\end{enumerate} 

\end{document}